\documentclass[12pt,a4paper,reqno]{amsart}

\usepackage[utf8]{inputenc}

\usepackage{array,amsbsy,amscd,amsfonts,color,amssymb,amstext,amsmath,latexsym,amsthm,amscd
,multicol,graphicx,tikz}
\usepackage[english]{babel}
\usepackage{hyperref}

\numberwithin{equation}{section}

\makeindex

\newtheorem{theorem}{Theorem}
\newtheorem{proposition}[theorem]{Proposition}
\newtheorem{lemma}[theorem]{Lemma}

\theoremstyle{definition}
\theoremstyle{definition}\newtheorem{definition}[theorem]{Definition}
\theoremstyle{definition}\newtheorem{remark}[theorem]{Remark}
\theoremstyle{definition}\newtheorem{example}[theorem]{Example}
\theoremstyle{definition}
\theoremstyle{definition}\newtheorem{assumption}[theorem]{Assumption}
\theoremstyle{definition}

\numberwithin{theorem}{section}

\def\proofof [#1] {\noindent {\bf Proof of #1. } }

\def\v #1.{\mathord{\raise 3pt\hbox{\mathsurround=0pt $\mathop\vee\limits^{#1}$\mathsurround=5pt}}}

\newcommand{\A}{\mathcal{A}}
\newcommand{\B}{\mathcal{B}}

\newcommand{\I}{\mathcal{I}}
\newcommand{\K}{\mathcal{K}}
\renewcommand{\H}{\mathcal{H}}
\newcommand{\M}{\mathcal{M}}
\newcommand{\CN}{\mathcal{N}}
\newcommand{\F}{\mathcal{F}}

\newcommand{\Ring}{\mathcal{R_A}}
\newcommand{\RRing}{\tilde{\mathcal{R}}_{\A}}

\newcommand{\N}{\mathbb{N}}
\newcommand{\NN}{\mathbb{N}_0}
\newcommand{\R}{\mathbb{R}}
\newcommand{\C}{\mathbb{C}}

\newcommand{\Z}{\mathbb{Z}}

\newcommand{\E}{\mathcal{E}}
\newcommand{\unit}{\mathbf{1}}

\newcommand{\bp}{\begin{proof}}
\newcommand{\ep}{\end{proof}}
\newcommand{\bdp}{\begin{dproof}}
\newcommand{\edp}{\end{dproof}}
\newcommand{\ra}{\rightarrow}

\newcommand{\Uone}{\operatorname{U}(1)}

\newcommand{\CAR}{\operatorname{CAR}}
\newcommand{\Loop}{\operatorname{L}}
\newcommand{\Diff}{\operatorname{Diff}(S^1)}

\newcommand{\PSL}{\operatorname{PSL(2,\R)}}

\newcommand{\PSI}{\operatorname{PSL(2,\R)}^{(\infty)}}

\newcommand{\Mat}{\operatorname{M}}
\newcommand{\Ext}{\operatorname{Ext}}

\newcommand{\Vir}{\operatorname{Vir}}
\newcommand{\tr}{\operatorname{tr }}
\newcommand{\Ad}{\operatorname{Ad }}

\newcommand{\rmd}{\operatorname{d}}
\newcommand{\rmi}{\operatorname{i}}

\newcommand{\rme}{\operatorname{e}}

\newcommand{\red}{\operatorname{red}}

\newcommand{\de}{\delta}

\newcommand{\pR}{{\pi_R}}

\newcommand{\sgn}{\operatorname{sgn }}

\newcommand{\dom}{\operatorname{dom }}

\newcommand{\kone}{\mathfrak{K}_{G_{\ell}}^1}

\newcommand{\kG}{\mathfrak{K}_{G_{\ell}}}

\newcommand{\tkone}{\tilde{\mathfrak{K}}_{G_{\ell}}^1}
\newcommand{\tkG}{\tilde{\mathfrak{K}}_{G_{\ell}}}
\newcommand{\kA}{\mathfrak{K}_\A}

\newcommand{\BG}{\operatorname{B}_{G_{\ell}}}

\renewcommand{\ker}{\operatorname{ker }}

\renewcommand{\dim}{\operatorname{dim }}

\newcommand{\Cci}{C^{\infty}}

\newcommand{\diag}{\operatorname{diag}}
\newcommand{\ind}{\operatorname{ind }}

\renewcommand{\lg}{\mathfrak{g}}

\newcommand{\id}{\operatorname{id}}

\newcommand{\ie}{{i.e.,\/}\ }

\newcommand{\eg}{{e.g.\/}\ }
\newcommand{\cf}{{cf.\/}\ }

\author{Sebastiano Carpi}
\thanks{S.C. is supported in part by the ERC advanced grant 669240 QUEST ``Quantum Algebraic
Structures and Models'' and GNAMPA-INDAM}
\address{Dipartimento di Economia, Universit\`a di Chieti-Pescara ``G. d'Annunzio'', Viale Pindaro, 42, 65127 Pescara, Italy\\
E-mail: {\tt s.carpi@unich.it}}
\author{Robin Hillier} 
\address{Department of Mathematics and Statistics, Lancaster University, Lancaster LA1 4YF, UK\\
E-mail: {\tt r.hillier@lancaster.ac.uk}}

\subjclass[2010]{46L87, 46L80, 19K35, 22E67, 81T05. Keywords: conformal nets, fusion ring, spectral triples, JLO cocycles, K-theory.}
\title{Loop Groups and Noncommutative Geometry}

\date{8 October 2018}

\begin{document}

\begin{abstract}
We describe the representation theory of loop groups in terms of K-theory and noncommutative geometry. This is done by constructing suitable spectral triples associated with the level $\ell$ projective unitary positive-energy representations of any given loop group $\Loop G$. The construction is based on certain supersymmetric conformal field theory models associated with $\Loop G$ in the setting of conformal nets. We then generalize the construction to many other rational chiral conformal field theory models including coset models and the moonshine conformal net.
\end{abstract}

\maketitle

\section{Introduction}

Since its foundation by Connes \cite{Con85}, \emph{noncommutative geometry} has been of growing importance, with impact on various fields of mathematics and physics: differential geometry, algebraic topology, index theory, quantum field theory, quantum Hall effect descriptions, etc., to name a few. 
The core idea is to work with an algebraic approach, which is in some sense ``dual'' to a topological one, namely instead of a given locally compact topological space one works with the commutative C*-algebra of continuous functions on it vanishing at infinity. It turns out that this offers several different and useful tools. According to Gelfand-Neimark's theorem, every commutative C*-algebra is actually of this type. To make things meaningful, usually some additional structure is requested to be given alongside the C*-algebra of continuous functions. The step to the noncommutative setting consists now basically in still requiring this further structure but allowing for arbitrary noncommutative C*-algebras (or even more general noncommutative topological algebras) instead of the commutative C*-algebra of continuous functions (\cf \cite{Con94,GBVF} for overview and a comprehensive study).

Examples of such further structures in noncommutative geometry are spectral triples $(A,(\pi,\H),D)$, where apart from the algebra $A$, a representation $\pi$ on some Hilbert space $\H$ and a selfadjoint operator $D$ with compact resolvent on $\H$ are given. In the commutative case such a triple together with some additional data completely describes a smooth compact manifold, according to Connes's reconstruction theorem \cite{Con13}. In the noncommutative case, this is by far not enough to understand the complete structure well, but it suffices in order to partially understand the noncommutative geometry of given objects, and to compute K-homology classes, noncommutative Chern characters and index pairings with K-theory. There is a bivariant version of K-theory, called KK-theory: a bifunctor from the category of C*-algebras to abelian groups. KK-theory plays a fundamental role in the structure theory of C*-algebras and in noncommutative geometry; apart from the inherent group addition, it admits a so-called intersection product, and it generalizes both K-theory and K-homology. The operator algebraic nature of noncommutative geometry has enabled many fruitful connections to other areas in mathematics and physics (\cf again \cite{Con94,GBVF}). In this article we establish a link with  the representation theory of loop groups.

\emph{Loop groups} are well-studied examples of infinite-dimensional Lie groups \cite{Milnor,PS}. Given a smooth compact manifold $X$ and a connected simply connected compact simple Lie group $G$ with Lie algebra $\lg$, a natural object to investigate is the group 
$\Cci(X,G)$ of smooth maps $X\ra G$, with point-wise multiplication and with the $C^\infty$ topology (uniform convergence of all partial derivatives). It is an infinite-dimensional Lie group modeled on the topological vector space $\Cci(X,\lg)$ of smooth maps $X\ra \lg$. 
Moreover, the Lie algebra of $\Cci(X,G)$ turns out to be $\Cci(X,\lg)$ with point-wise brackets \cite[Example 1.3]{Milnor}. Loop groups are obtained in the special case where $X=S^1$. Accordingly the loop group of $G$ is given by $\Loop G := \Cci(S^1,G)$ and it is an infinite-dimensional Lie group with Lie algebra $\Cci(S^1,\lg)$ \cite{PS}. The latter admits nontrivial central extensions corresponding to the affine Kac-Moody Lie algebras associated with $\lg$ \cite{Kac,PS}. 

Loop groups are important objects to study in mathematics: they play a fundamental role in conformal quantum field theory and string theory. They have shown deep relations with various other mathematical areas such as number theory, subfactor theory, quantum groups and topological quantum field theories, see e.g. \cite{BK,EK1998,Frenkel,Fuchs,Gannon,WasICM}.

Loop groups also have a very interesting representation theory. 
It is obtained by restricting to a special class of (projective) unitary representations, namely the so-called positive-energy representations, or more precisely, those arising by integrating unitary highest weight irreducible representations of the underlying affine Kac-Moody algebra $\hat{\lg}_\C$. There is a distinguished central element $c_\lg$ in $\hat{\lg}_\C$ and it takes positive integer value in such a representation; this integer is called the level of the representation. For every given level $\ell$ there are only finitely many classes of such representations; remarkably, they generate a 
commutative ring $R^{\ell}(\Loop G)$, the so-called Verlinde fusion ring. The ring structure comes from the operator product expansion in the conformal field theory model associated to the representation theory of $\Loop G$ at level $\ell$.  Mathematically this can be described in various ways. For example the ring product can be defined through the so called Verlinde formula from the modular invariance property of the characters of the representations, see e.g. \cite[Chap.5]{Fuchs} and \cite[Chap. 6]{Gannon}. 
More conceptually $R^{\ell}(\Loop G)$ can be defined as the set of equivalence classes of objects in the braided tensor category of level $\ell$ representations of 
$\Loop G$. The tensor structure on the latter category can be obtained e.g. through the Huang-Lepowsky theory of tensor products for vertex operator algebra modules \cite{HL1999}, see also \cite[Chap. 6]{Gannon}. For explicit computations of the fusion rings see \cite{Douglas}.

In this paper we give a description of the Verlinde fusion ring  $R^{\ell}(\Loop G)$ in terms of K-theory and noncommutative geometry for any given connected simply connected compact simple Lie group $G$ and positive integer level $\ell$. 
The first step in this direction is the definition of a natural and canonical universal C*-algebra $\kG$  for the level $\ell$ positive-energy representation theory of $\Loop G$. To this end we use the theory of conformal nets and follow the ideas in \cite{CCHW}. More precisely, we have to assume that the conformal net $\A_{G_{\ell}}$ associated with the representation theory of $\Loop G$ at level $\ell$ is completely rational and that the corresponding ring generated by the Doplicher-Haag-Roberts (DHR) endomorphisms is isomorphic to $R^{\ell}(\Loop G)$. This assumption is known to be satisfied e.g. for $G=\mathrm{SU}(n)$ at any positive integer level, and it is widely expected to be true in general although this is still a very important open problem which goes  far beyond the scope of this paper, see e.g. \cite[Problem 3.32]{Kaw2015} and the discussion following Assumption \ref{CFT-assumption} here below. We will make this assumption throughout the paper. The universal C*-algebra $\kG$ is then defined to be the \emph{compact universal C*-algebra} 
$\mathfrak{K}_{\A_{G_{\ell}}}$ introduced in \cite{CCHW} which is a natural (nonunital) universal C*-algebra for the representation theory of the net $\A_{G_{\ell}}$ and hence for the level $\ell$ positive-energy representation theory of $\Loop G$. It then follows from the results in \cite{CCHW} that there is a group isomorphism of $R^{\ell}(\Loop G)$ onto the K-theory group  $K_0(\kG)$. The latter can be seen as an analogue of the isomorphism $R(G) \simeq K_0(C^*(G))$ where $R(G)$ is the representation ring of $G$ and $C^*(G)$ is the group C*-algebra, which actually holds true for every compact group $G$.   

Then we define a  suitable norm dense subalgebra $\kone\subset \kG$ and, for any irreducible level $\ell$ representation of $\Loop G$, we construct a spectral triple with algebra $\kone$ and show that the entire cyclic cohomology class of the corresponding JLO cocycle completely determines the unitary equivalence class of the loop group representation. These spectral triples give rise to a group isomorphism from the level $\ell$ Verlinde fusion ring $R^{\ell}(\Loop G)$ of $\Loop G$ onto the K-homology group $K^0(\kG)$ which is dual through the index pairing to the isomorphism $R^{\ell}(\Loop G) \simeq K_0(\kG)$. 

The construction of the above spectral triples is based on the supersymmetric chiral conformal field theory  (CFT) models associated to the loop group representations. The above construction naturally identifies the additive group underlying the fusion ring  
$R^{\ell}(\Loop G)$ with the  K-groups $K^0(\kG)$ and $K_0(\kG)$. A K-theoretical description of the fusion product can know be given following the ideas in \cite{CCHW,CCH}. The DHR endomorphisms  of the net $\A_{G_{\ell}}$ give rise to endomorphisms of the C*-algebra $\kG$ and, as a consequence, an injective ring homomorphism from $R^{\ell}(\Loop G)$ into $KK(\kG,\kG) \simeq \mathrm{End}(K_0(\kG)) \simeq \mathrm{End}(K^0(\kG))$ so that the fusion product can be naturally described in terms of the Kasparov product in KK-theory.

These results are deeply related to the noncommutative geometrization program for CFT recently developed in \cite{Lo01,CKL,CHL,CHKL,CHKLX}, cf. also \cite{RV2014a,RV2014b} for related work. Actually, this program was one of the initial motivations for the present article. We obtain here for the first time examples for which this noncommutative geometrization program can be carried out for all sectors of the corresponding conformal nets. This means that all irreducible sectors can be separated by JLO cocycles associated to spectral triples naturally arising from supersymmetry.  
The central new idea which allows us to overcome certain technical difficulties found \eg in \cite{CHKLX,CHL} is the one to consider spectral triples associated to degenerate representations. As a consequence we can relax the requirement of superconformal symmetry by considering superconformal tensor product extensions, the {\em superconformal tensor product} considered in Section  \ref{sec:nonsymplyconnected}. By the results in the latter section these examples include quite a large family  of completely rational conformal nets beyond loop group conformal nets. Another consequence of the results of this paper is that the noncommutative geometric description of the representation theory of these conformal nets is directly related to the K-theoretic description recently given in \cite{CCH,CCHW}.

A different K-theoretical description of the fusion ring $R^{\ell}(\Loop G)$  has been given by Freed, Hopkins and Teleman (FHT)
\cite{FHT2008,FHT2011a,FHT2011b,FHT2013}, cf. also \cite{Mick2004,Mick2005} for related work. In those papers the Verlinde fusion ring of $\Loop G$ at level $\ell$ has been identified with a twisted version of the equivariant topological K-theory group $K_G(G)$, where $G$ acts on itself by conjugation. The twisting is determined by the level $\ell$. Under this identification the fusion product corresponds to the convolution (Pontryagin) product on the twisted $K_G(G)$ and the latter can be defined without direct reference to the Verlinde fusion product in $R^{\ell}(\Loop G)$ \cite[Thm.1]{FHT2011b},  
see also \cite{FHT2011a}.
Very interesting relations of this result with subfactor theory and modular invariants have been investigated by Evans and Gannon \cite{Evans2006,EG2009,EG2013}.

Our results indirectly show the identification of the above twisted equivariant K-theory of $G$ with the K-theory group $K_0(\kG)$ of the C*-algebra $\kG$. Moreover there are various structural similarities between our approach and FHT. Our Dirac operators are essentially the same as those considered in \cite[Sec.11]{FHT2011b}, see also \cite{FHT2013}. In our case the Dirac operators are used to construct 
spectral triples with algebra $\kone$ in order to obtain K-homology classes $K^0(\kG)$. In the FHT case the Dirac operators are used to construct equivariant Dirac families which give rise to twisted equivariant K-theory classes.

We believe it would be very interesting to exploit this relationship deeper and more explicitly in future. In particular we believe that a better understanding of this relationship would shed more light on these K-theoretical approaches to the representation theory of loop groups and in particular to its deep mathematical relations with conformal field theory. More generally, it could reveal new connections between twisted K-theory, non-commutative geometry and subfactor theory. Moreover it could shed  some light on the relation of the FHT convolution product in equivariant twisted K-theory and the Kasparov product defined in \cite{TuXu}. From this point of view our analysis should not be seen as alternative to the FHT work but rather complementary to it. It should be pointed out that the FHT analysis goes beyond the case in which $G$ is a simply connected compact Lie group and actually it also deals with non-connected compact Lie groups (in particular with finite groups) and twisted loop groups. On the other hand we give in this paper an abstract version of our construction, formulated in terms of completely rational conformal nets, which applies to many conformal field theory models besides the case of loop groups. In particular we can cover the case of loop groups associated to arbitrary connected compact Lie groups but also coset models, minimal models and the conformal net analogue of the Frenkel-Lepowsky-Meurman moonshine vertex operator algebra  
\cite{FLM} constructed by Kawahigashi and Longo in \cite{KL06}. We should also point out that the K-theoretical description of the DHR fusion ring in \cite{CCH,CCHW} is very general since it works for all completely rational conformal nets.

Our paper is organized as follows. We start off with a preliminary section on noncommutative geometry for nonunital algebras, which is essential to understand our results and seems to be difficult to find in literature in this form. We also include some KK-theory basics at the end of that section.

In Section \ref{sec:DHR} we provide an introduction to loop groups and to conformal nets in general and explain how the relationship with KK-theory emerges naturally, \cf \cite{CCH,CCHW}. We then discuss the special case of loop group conformal nets, which provides the basis for the subsequent section.

The main Section \ref{sec:JLO} then deals with the construction of our noncommutative geometric objects: for a given connected simply connected compact simple Lie group $G$ and hence for its  corresponding loop group 
$\Loop G$, we construct Dirac operators and a global differentiable algebra $\kone$ on which the Dirac operators act; they give rise to spectral triples, noncommutative Chern characters (also known as JLO cocycles) and index parings with K-theory classes corresponding to characteristic projections in $\kone$. The underlying ideas in the construction are supersymmetry and conformal nets: tensoring the loop group by the (graded) CAR algebra, we turn the given representations into graded representations; for the latter special representations, the so-called super-Sugawara (or Kac-Todorov) construction \cite{KT}
guarantees supersymmetry and Dirac operators. The differentiable algebra $\kone$ comes out as a natural byproduct. 
Moreover, the theory of conformal nets explicitly identifies the fusion product as a Kasparov product. 

In Section \ref{sec:nonsymplyconnected} we introduce the notion of superconformal tensor product for conformal nets that can be considered as a generalization of the super-Sugawara construction for loop groups models. We show that our previous analysis for loop groups $\Loop G$ with $G$ 
simply connected, extends in a rather straightforward way to the case of completely rational conformal nets with superconformal tensor product.
We then provide many examples of completely rational conformal nets admitting a superconformal tensor product as announced above.

\section{Entire cyclic cohomology for nonunital Banach algebras}\label{sec:prelim}

Let $(A,\|\cdot\|)$ be a Banach algebra. If $A$ is nonunital, its \emph{unitalization} $\tilde{A}$ is obtained by adjoining a unit $\unit_{\tilde{A}}$ to $A$, with multiplication defined by
\[
(a_1+t_1\unit_{\tilde{A}})\cdot (a_2+t_2\unit_{\tilde{A}}) := a_1a_2+t_1a_2+t_2a_1 + t_1t_2\unit_{\tilde{A}}, 
\]
and norm $\|a+t\unit_{\tilde{A}}\|^\sim :=\|a\|+|t|$, for every $a+t\unit_{\tilde{A}}\in\tilde{A}$. Then $A\subset \tilde{A}$ is a closed ideal and $\tilde{A}/A\simeq\C$. If $\|\cdot\|$ was a C*-norm, one can construct a C*-norm on $\tilde{A}$ equivalent to $\|\cdot\|^\sim$, \cf \cite[II.1.2]{Bl2}. However, we will not need this. Moreover, we shall drop the superscript ``$\sim$'' of $\|\cdot\|^\sim$ henceforth; analogously for subscripts of $\|\cdot\|$ indicating the algebra whenever they appear and confusion is unlikely. If $\pi:A\ra B(\H_\pi)$ is a representation, then it extends to a unital representation $\tilde{\pi}$ of $\tilde{A}$ by $\tilde{\pi}(a+t\unit_{\tilde{A}}):= \pi(a) + t\unit_{\H_\pi}$ and we shall always work with this extension. 

For every $r\in\N$, let $\Mat_r(A)$ be the algebra of $r\times r$ matrices with entries in $A$. 
The maps $x\mapsto \diag(x,0)$ define natural embeddings $\Mat_r(A) \ra \Mat_{r+1}(A), r\in \N$. Then  $\Mat_\infty(A)$ is defined as the corresponding inductive limit, \ie the algebra of infinite-dimensional matrices with entries in $A$ and all but finitely many ones of them zero \cite[5.1]{Bla}. 
The algebras $M_r(A), r \in \N$, can be made into  Banach algebras in many equivalent ways. 
One can choose the norms on $M_r(A)$ in such a way that the embeddings are isometries. Here we fix any sequence of such norms. It induces a norm on $M_\infty(A)$ which is accordingly made into a normed algebra, too. What follows in this paper does not depend on the particular choice of the norms. 

$A$ is called \emph{stably-unital} if $\Mat_\infty(A)$ has an approximate identity of idempotents; this is \eg the case if $A$ has an approximate identity of idempotents \cite[5.5]{Bla}.

Entire cyclic cohomology is usually studied in the setting of unital Banach algebras. A definition for nonunital ones using the cyclic cocomplex can be found in detail in \cite[Sect.6]{Cu}. The subsequent one following \cite[Sect.2]{Kha} is stated in the $(b,B)$-bicomplex, which is more suitable in our setting since we want to apply it to the JLO cochain. The two definitions are equivalent according to \cite[Sect.2]{Kha} together with \cite[Prop.4.2]{LQ}.

\begin{definition}\label{def:ECC}
(i) Let $A$ be a nonunital Banach algebra and, for any nonnegative integer $n$, let $C^n(A)$ be the vector space of reduced $(n+1)$-linear forms $\phi_n$ on the unitalization $\tilde{A}$, \ie such that
\begin{itemize}
\item[-] $\phi_0(\unit_{\tilde{A}})=0$,
\item[-] $\phi_n(a_0,a_1, ...,a_n)=0$ if $a_i= \unit_{\tilde{A}}$ for some $1\leq i \leq n$ (\emph{simplicial normalization} \cite{Brodzki}).
\end{itemize} 
For integers $n < 0$ we set $C^n(A) := \{0\}$. Let $C^\bullet(A) := \prod_{n=0}^\infty C^n(A)$ be the space of sequences $\phi= (\phi_n)_{n \in \N_0}$, with $\phi_n \in C^n(A)$ and define the operators $b:C^\bullet(A) \to C^\bullet(A)$ and $B:C^\bullet(A) \to C^\bullet(A)$ by 

\begin{align*}
(b\phi)_n (a_0, ..., a_n) :=& \sum_{j=0}^{n-1} (-1)^j \phi_{n-1}(a_0, ..., a_j a_{j+1}, ..., a_n)\\
& + (-1)^n \phi_{n-1}(a_n a_0, a_1,  ..., a_{n-1}),\\
(B\phi)_n(a_0, ..., a_n)  := & \sum_{j=0}^n (-1)^{jn} \phi_{n+1}(\unit_{\tilde{A}}, a_j, ..., a_n, a_0, ..., a_{j-1}), \quad a_0,\ldots,a_n \in \tilde{A}.
\end{align*}

The linear map $\partial: C^\bullet(A) \ra C^\bullet (A)$ defined by $\partial := b + B$ satisfies $\partial^2=0$ and, with 
 the coboundary operator $\partial$, $C^\bullet(A)$ becomes the \emph{cyclic cocomplex} $C^\bullet(A) = (C^e(A),C^o(A))$ over $\Z/2\Z$, namely the elements of  $C^e(A)= \prod_{n=0}^\infty C^{2n}(A)$ (the {\it even cochains})
are mapped into the elements of $C^o(A)= \prod_{n =0}^\infty C^{2n+1}(A)$ (the {\it odd cochains}) and vice versa. The elements $\phi\in C^\bullet(A)$ satisfying $\partial \phi=0$ are called the \emph{cyclic cocycles of $A$}.

(ii) A cochain $\phi= (\phi_n)_{n\in\NN} \in C^\bullet(A)$ is called \emph{entire} if
\[
 |\phi_n(a_0,...,a_n)| \le \frac{1}{\sqrt{n!}} \|a_0\|\cdots\|a_n\|, \quad a_i\in \tilde{A}, n\in\NN.
\]
Letting $CE^\bullet(A)$ be the entire elements in $C^\bullet(A)$, one defines the \emph{entire cyclic cohomology} $(HE^e(A), HE^o(A))$ of $A$ as the cohomology of the cocomplex $((CE^e(A),CE^o(A)), {\partial})$. The cohomology class of an entire cyclic cocycle $\phi \in CE^\bullet(A) \cap \ker (\partial)$ will be denoted by $[\phi]$.
\end{definition}

We recall that in the case of unital $A$ instead no unitalization is considered; $C^n(A)$ stands then for the simplicially normalized (not necessarily reduced) cochains on $A$ itself, and entire cyclic cohomology is defined accordingly  \cite{Cu}.

\begin{definition}\label{def:K0}
\cite[Sect.5.5]{Bla} Let $A$ be a stably-unital Banach algebra. 
We denote by $P\Mat_r(A)$ the set of idempotents in $\Mat_r(A)$, $r\in \N \cup \{\infty\}$. An equivalence relation on $P\Mat_\infty(A)$ is defined by $p \sim q$ if there are $x, y \in \Mat_\infty(A)$ such that $p=xy$, $q=yx$. There is a binary operation 
\[
(p_1,p_2)\in P\Mat_{r_1}(A) \times P\Mat_{r_2}(A)  \mapsto  p_1\oplus p_2 := \diag(p_1, p_2) \in P\Mat_{r_1+r_2}(A),
\]
which turns $P\Mat_\infty(A)/\sim$ into an abelian semigroup. 
Then the \emph{$K_0$-group of $A$} is defined as
\[
  K_0(A) := \textrm{Grothendieck group of } P\Mat_\infty(A)/\sim.
\]
We write $[p]$ for the element in $K_0(A)$ induced by a projection $p\in A$.
\end{definition}

\begin{definition}\label{def:SpTr}
A \emph{$\theta$-summable even spectral triple} is a triple $(A,(\pi,\H),D)$, where
\begin{itemize}
\item[-] $A$ is an algebra;
\item[-] $\H$ is a Hilbert space graded by a selfadjoint unitary $\Gamma$, and $\pi$ is a representation of $A$ on $\H$ commuting with $\Gamma$;
\item[-] $D$ is a self-adjoint operator with $\rme^{-tD^2}$ trace-class for all $t>0$, with $\Gamma D \Gamma =-D$, and with $\pi(A)\subset \dom(\delta_D)$, where $\delta_D$ is the derivation on $B(\H)$ induced by $D$ as explained below.
\end{itemize}
\end{definition}
Note that, according to the above definition, if $A$ is nonunital then $(A,(\pi,\H),D)$ is a $\theta$-summable even spectral triple
if and only if $(\tilde{A},(\tilde{\pi},\H),D)$ is. 
Henceforth we shall deal with spectral triples in the case where $A$ is a nonunital Banach algebra.

Given a self-adjoint operator $D$ on $\H$, one associates a derivation $\delta_D$ of $B(\H)$ as follows: define $\dom(\delta_D)$ as the algebra of $x\in B(\H)$ such that
\[
x D \subset D x - y
\]
for some $y\in B(\H)$, and $\delta_D(x):=y$ in this case, \cf \cite[Sect.3.2]{BR}.

The \emph{JLO cochain} $\tau=(\tau_{n})_{n\in 2\NN}$ of a spectral triple $(A,(\pi,\H),D)$ with $A$ a nonunital Banach algebra is defined as $\tau:= \tilde{\tau}-\psi$, where
\begin{equation}\label{eq:JLO-def}
\begin{aligned}
 \tilde{\tau}_{n}(a_0,\ldots,a_n) =
\int_{0\le t_1\le ...\le t_n \le 1}\tr \Big( & {\Gamma} \tilde{\pi}(a_0) \rme^{-t_1 D^2} \delta_D(\tilde{\pi}(a_1)) \rme^{-(t_2-t_1) D^2} \cdots \\
& \cdots \, \delta_D(\tilde{\pi}(a_n))\rme^{-(1-t_n) D^2} \Big) \rmd t_1 \cdots \rmd t_n
\end{aligned}
\end{equation}
and
\[
\psi_n(a_0,...,a_n):= \begin{cases}
t \tilde{\tau}_0(\unit_{\tilde{A}}) & \mathrm{if} \; n=0,\; a_0 \in t\unit_{\tilde{A}} + A \\
0 & \mathrm{if} \; n>0,
\end{cases}
\]
for every $a_i\in \tilde{A}$ and every $n\in\NN$. Clearly, in restriction to entries in $A$, we have $\tau\restriction_A = \tilde{\tau}\restriction_A$. 

Notice that in the common context of unital algebras and spectral triples (\eg in \cite{CHL}), $\tilde{\tau}$ would be a simplicially normalized (but not reduced) cochain on $A$ itself, namely the original JLO cochain from \cite{JLO}.

Now, let $(A , (\pi,\H), D)$ be a $\theta$-summable even spectral triple with grading operator $\Gamma$. 
Let $\H_\pm = \ker (\Gamma \mp \unit_{B(\H)})$ so that we have the decomposition $\H= \H_+ \oplus \H_- $.
If $T$ is a densely defined operator on $\H$ which is odd, i.e., such that $\Gamma T \Gamma = -T$, 
then we can write 
$$T = \left( 
\begin{array}{cc}
0 & T_-  \\
T_+ & 0 
\end{array} \right), $$ 
with operators $T_\pm$ from (a dense subspace of) $\H_\pm$ to $\H_\mp$.
 Accordingly, if $(A , (\pi,\H), D)$ is a $\theta$-summable even spectral triple then for the selfadjoint $D$
we can write 
$$D = \left( 
\begin{array}{cc}
0 & D_-  \\
D_+ & 0 
\end{array} \right), $$ 
with $D_-= D_+^*$. On the other hand, if $T$ is even, i.e., it commutes with $\Gamma$, then we can write 
$$T = \left( 
\begin{array}{cc}
T_+ & 0  \\
0 & T_- 
\end{array} \right), $$ 
with operators $T_\pm$ from (a dense subspace of) $\H_\pm$ to $\H_\pm$.

For any positive integer $r$ we denote by $\pi_r$ the representation of $\Mat_r(A)$ on $\H_r:= \C^r \otimes \H$ defined by $\pi_r(m \otimes a) := m \otimes \pi(a)$, $m \in \Mat_r(\C)$, $a\in A$.  Moreover, for every operator $T$ on $\H$ we consider the operator $T_r := \unit \otimes T$ on $\H_r$. Then, for every $r\in \N$,  $(\Mat_r(A) , (\pi_r,\H_r), D_r)$ is a $\theta$-summable even spectral triple which is even with grading operator $\Gamma_r$.
For an $n$-linear form $\phi_{n-1}:\tilde{A}^{\otimes n}\ra\C$, we introduce the $n$-linear form $\phi_{n-1}^r$ on $(\Mat_r(\C)\otimes\tilde{A})^{\otimes n}= \Mat_r(\tilde{A})^{\otimes n}$ as
\[
 \phi_{n-1}^r (m_0 \otimes a_0, .... ,m_n\otimes a_n) :=  \tr (m_0...m_n)  \phi_{n-1}(a_0, .... ,a_n)
\] 
and linear extension.

The following fundamental theorem is an adaptation to the present nonunital setting of the corresponding unital setting, \cf \eg \cite{Con88,JLO,GS,CHL}.
 
\begin{theorem}\label{th:ECCJLO}
Let $(A,(\pi,\H),D)$ be a $\theta$-summable even spectral triple such that $A$ is a nonunital Banach algebra and the representation $\pi$ of $A$ in the Banach algebra $(\dom (\delta_D),\|\cdot\|+\|\delta_D(\cdot)\|_{B(\H)})$ is continuous.
\begin{itemize}
\item[(i)] The cochain $(\tau_{n})_{n\in 2\NN}$ is an even entire cyclic cocycle on $A$. We call it  the \emph{JLO cocycle} or \emph{Chern character} of $(A,(\pi,\H),D)$.
\item[(ii)] The values of the maps  $(\phi,p) \in \left( CE^e(A) \cap \ker (\partial)  \right) \times P\Mat_r(A) \mapsto \phi (p) \in \C$, 
$r \in \N$, where 
\[ 
\phi (p) := \phi_0^{r}(p) + \sum_{k=1}^\infty (-1)^k \frac{(2k)!}{k!} \phi_{2k}^{r}((p-\frac12),p,...,p) , 
\]
only depend on the cohomology class $[\phi]$ of the entire cyclic cocycle $\phi$ and on the $K_0(\tilde{A})$-class $[p]$ of the idempotent $p$ and are additive on the latter. If $A$ is stably-unital, this gives rise to a pairing $\langle [\phi], [p] \rangle := \phi(p)$ between the even entire cyclic cohomology $HE^e(A)$ and K-theory $K_0(A)$. Moreover, the operator
$\pi_r(p)_{-} {D_r}_{+} \pi_r(p)_+$ from $\pi_r(p)_+{\H_r}_+$ to $\pi_r(p)_-{\H_r}_-$ is a Fredholm  operator and for the even JLO cocycle 
$\tau$ we have
\[
\tau (p) = \langle [\tau], [p] \rangle = \ind_{\pi_r(p)_+{\H_r}_+}
\big( \pi_r(p)_{-} {D_r}_{+} \pi_r(p)_+ \big) \in \Z .
\]
\end{itemize}
\end{theorem}

\begin{proof}
For unital algebras and spectral triples this theorem is well-known, namely the JLO cochain $\tilde{\tau}$ is an even entire cyclic cocycle on the unital algebra $\tilde{A}$ \cite{Con94,GS,JLO}. The same is true for $\psi$ as is easily verified. $\tau=\tilde{\tau}-\psi$ is again an entire cyclic cocycle of $\tilde{A}$ because both $\tilde{\tau}$ and $\psi$ are so. Evenness is obvious, \ie $\tau_n=0$ if $n\in\NN$ is odd. For every $n\in\NN$, $\tau_n$ is reduced, \ie $\tau_0(\unit_{\tilde{A}})=0$ and $\tau_n(a_0,a_1, ...,a_n)=0$ if $a_i= \unit_{\tilde{A}}$ for some $1\leq i \leq n$. Thus $\tau$ is an even entire cyclic cocycle of $A$, proving (i). 

Part (ii) becomes clear, too, by considering first all the statements and the well-known pairing between entire cyclic cohomology and K-theory (\cf \eg \cite{Con88,Con94,JLO,GS,CHL}) for the unital algebra $\tilde{A}$, namely: $\tau$ pairs with $K_0(\tilde{A})$ as
\[ 
\tau (p) := \tau_0^{r}(p) + \sum_{k=1}^\infty (-1)^k \frac{(2k)!}{k!} \tau_{2k}^{r}((p-\frac12),p,...,p) 
\]
if $p\in\Mat_r(\tilde{A})$ represents a class in $K_0(\tilde{A})$, and $\tau(p_1)=\tau(p_2)$ for projections $p_1,p_2\in\Mat_\infty(\tilde{A})$ if $p_1$ and $p_2$ belong to the same class in $K_0(\tilde{A})$. Finally, if $A$ is stably-unital, then $K_0(A)$ is defined according to Definition \ref{def:K0} using equivalence classes of projections in $\Mat_\infty(A)$; if moreover $p_1,p_2\in\Mat_\infty(A)\subset\Mat_\infty(\tilde{A})$ and give rise to the same class in $K_0(A)$ then according to \cite[5.5.2]{Bla} they have the same class in $K_0(\tilde{A})$, so $\tau(p_1)=\tau(p_2)$ according to the previous step. Thus the pairing of $\tau$ with $K_0(\tilde{A})$ restricts to a pairing with $K_0(A)$. The Fredholm property and the final formula follow directly from the unital case by restriction.
\end{proof}

In the following sections we will look at this kind of pairing from the point of view of Kasparov's KK-theory, for which we are now providing a few preliminaries based on the summary in \cite{CCH}. A thorough introduction with proofs can be found in \cite[Ch.17\&18]{Bla}.

Let now $A,B$ be stably-unital separable C*-algebras, then a \emph{Kasparov $(A,B)$-module} is a tuple $(\E,\phi,F)$, where $\E$ is a countably generated $\Z_2$-graded Hilbert $B$-module, $\phi:A\ra B(\E)$ is a graded *-homomorphism, and $F\in B(\E)$ has degree one, such that
\begin{equation}\label{eq:Fredholm}
  (F-F^*)\phi(a), \quad (F^2-\unit) \phi(a), \quad [F,\phi(a)]
\end{equation}
lie all in the compact operators $K(\E)$ (the norm closure of finite-rank operators) on $\E$, for all $a\in A$. 
With a suitable concept of homotopy, one defines $KK(A,B)$ as the set of homotopy equivalence classes of Kasparov $(A,B)$-modules. E.g. for every unitary $u\in B(\E)$ of degree zero,  $(\E,\Ad(u)\circ\phi,\Ad(u)(F))$ is again a Kasparov $(A,B)$-module, which is homotopy equivalent to $(\E,\phi,F)$, and therefore the two give rise to the same element in $KK(A,B)$. We write $\H_A$ for the standard right Hilbert $A$-module $A\otimes l^2(\Z)$ and $\hat{\H}_A$ for the corresponding $\Z/2$-graded one $\H_A\oplus\H_A$ with grading $\unit\oplus-\unit$.
 
For our immediate purposes in the following sections, the most relevant facts about KK-theory can be summarized as follows, with $A,B,C$ stably-unital separable C*-algebras:
\begin{itemize}
\item[$(1)$] There is a direct sum for Kasparov $(A,B)$-modules, which passes to the quotient $KK(A,B)$ and turns $KK(A,B)$ into an abelian group.
\item[$(2)$] There is a canonical identification of $KK(\C,A)$ with the K-theory group $K_0(A)$ (as additive groups). Similarly, there is a canonical identification of $KK(A,\C)$ with the K-homology group $K^0(A)$.\\
This identification works as follows for $K_0(A)$: 
\[
[p_+]-[p_-] \in K_0(A) \mapsto
\Big[ \hat{\H}_A, \phi_{p_+} \oplus \phi_{p_-}, \begin{pmatrix}
0 & \unit \\ \unit & 0
\end{pmatrix} \Big] \in KK(\C,A),
\]
as any element in $K_0(A)$ may be written as a formal difference $[p_+]-[p_-]$, with (not unique) $p_+,p_- \in A\otimes \K$, and $\phi_p: t\in\C \mapsto tp \in A\otimes \K$, for any projection $p\in A\otimes \K$. Owing to this identification, we may consider $[p]$ as an element of $KK(\C,A)$, for every $p\in A\otimes K$, \cf \cite[17.5-17.6]{Bla}\\
On the other hand, $K^0(A)$ is by definition the group generated by homotopy classes of even Fredholm modules on a standard $\Z_2$-graded separable Hilbert space $\hat{\H}$, \ie classes $(\hat{\H},\phi,F)$ with $\phi:A\ra B(\hat{\H})$ a graded *-homomorphism and $F$ a unitary selfadjoint operator on $B(\hat{\H})$ of degree 1 such that the graded commutator satisfies $[F,\phi(a)]_{(+)}\in K(\hat{\H})$, for all $a\in A$, \cf \cite[p.294]{Con94}. Thus $K^0(A)$ coincides with $KK(A,\C)$.
\item[$(3)$] Every *-homomorphism $\phi:A\ra B$ naturally defines a $KK(A,B)$-element $\{\phi\}$
as the homotopy class of $(B,\phi,0)$, where we have identified $B(B)$ with 
the multiplier algebra $\M(B)$ of $B$. $\{\phi\}$ depends only on the unitary-equivalence class in $\M(B)$, \cf \cite[17.1-17.2]{Bla}.
\item[$(4)$] There exists a well-defined bilinear map $\times$, the so-called Kasparov product 
\[
  KK(A,B)\times KK(B,C) \ra KK(A,C),
\]
which is associative. It is in general complicated to define and we refer to \cite[Ch.18]{Bla}, but for the following special cases we provide formulae.
\item[(5)] Suppose $A,B,C$ are trivially graded and given two classes of Kasparov modules
\[
\Big[\hat{\H}_B,\psi\oplus 0,\begin{pmatrix}
0 & \unit \\ \unit & 0
\end{pmatrix} \Big] \in KK(A,B), \quad
\Big[\hat{\H}_C,\phi\oplus 0,\begin{pmatrix}
0 & \unit \\ \unit & 0
\end{pmatrix} \Big] \in KK(B,C),
\]
whith $\psi$ a *-homomorphism from $A$ to $B\otimes\K$ and $\phi$ is a *-homomorphism from $B$ to $C\otimes\K$; the latter extends to $B\otimes \K$, denoted again by $\phi$. As explained in \cite[Ex.18.4.2(c)]{Bla}, the Kasparov product of the two is then given by the element
\[
\Big[\hat{\H}_B,\psi\oplus 0,\begin{pmatrix}
0 & \unit \\ \unit & 0
\end{pmatrix} \Big]\times
\Big[\hat{\H}_C,\phi\oplus 0,\begin{pmatrix}
0 & \unit \\ \unit & 0
\end{pmatrix} \Big] =
\Big[\hat{\H}_C,\phi\circ\psi\oplus 0,\begin{pmatrix}
0 & \unit \\ \unit & 0
\end{pmatrix} \Big]
\]
in $KK(A,C)$. 
\item[$(6)$] If $\psi : A \to B$ and $\phi : B \to C$ are *-homomorphisms then 
\begin{equation}
\{\psi\} \times \{\phi\} = \{\phi\circ\psi\} .
\end{equation}
If  $\;\id_A$  is the identity automorphism of $A$ then $\{\id_A\}$ is the neutral element in $KK(A,A)$ for the Kasparov product.
Hence $KK(A,A)$ is a unital ring, \cf \cite[18.7.1]{Bla}.
\end{itemize}
\medskip

The Kasparov product $KK(\C,A)\times KK(A,\C)\ra KK(\C,\C)\simeq\Z$ gives rise to an index pairing  between $K_0(A)=KK(\C,A)$ and 
$K^0(A)=KK(A,\C)$ and to a corresponding map $\gamma_a : K^0(A) \to \mathrm{Hom}\big(K_0(A),\Z \big)$.
 Moreover, the Kasparov product  $KK(\C,A)\times KK(A,A)\ra KK(\C,A) = K_0(A)$ gives rise to a map 
$\gamma_b: KK(A,A) \to \mathrm{End}\big({K_0(A)}\big)$. Similarly, the Kasparov product $KK(A,A)\times KK(A,\C)\ra KK(A,\C) = K^0(A)$ gives 
rise to a map $\gamma_c: KK(A,A) \to \mathrm{End}\big({K^0(A)}\big)$.

\section{Loop group representations, conformal nets and K-theory}\label{sec:DHR}

Let $G$ be a connected simply connected compact simple Lie group and denote its Lie algebra by $\lg$ and its dimension by $d\in\N$, and let $\Loop G := \Cci(S^1,G)$ the corresponding loop group. 
It is an infinite-dimensional Lie group, with Lie algebra the \emph{loop algebra} $\Loop \lg := \Cci(S^1,\lg)$. We denote by $\Loop\lg_\C$ the complexification of $\Loop \lg$. The Lie subalgebra $\tilde{\lg}_\C \subset \Loop \lg_\C $ consisting of elements with finite Fourier series admits a nontrivial central extension $\hat{\lg}_\C$ called the \emph{affine Kac-Moody algebra} associated with $\lg$, see \cite{Fuchs,Kac,PS}.
We shall denote the corresponding distinguished central element by $c_{\lg}$. 

We would like to consider positive-energy representations of $\Loop G$ from the point of view of noncommutative geometry. A strongly continuous projective unitary representation 
$\lambda:\Loop G\ra U(\H_\lambda)/\Uone $ on a Hilbert space $\H_\lambda$ is of \emph{positive energy} if
there is a strongly continuous one-parameter group $U^\lambda: \R \ra U(\H_\lambda)$ whose self-adjoint generator 
$L_0^\lambda$ (the {\it conformal Hamiltonian}) has non-negative spectrum and such that
\begin{equation}
\label{Eq:positive-energy}
U^\lambda(t)\lambda(g)U^\lambda(t)^* = \lambda(g_t), \quad g\in\Loop G,
\end{equation}
where $g_t$ is defined by $g_t(z) := g(\mathrm{e}^{-\mathrm{i}t}z)$, $z\in S^1$. Eq. (\ref{Eq:positive-energy}) should be understood 
in the projective sense so that if for every $g \in \Loop G$, $\hat{\lambda}(g) \in U(\H_\lambda)$ denotes a given choice of a representative 
of $\lambda(g) \in U(\H_\lambda)/\Uone$ then 
\begin{equation}
\label{Eq:positive-energy_b}
U^\lambda(t)\hat{\lambda}(g)U^\lambda(t)^* = \chi(g,t)\hat{\lambda}(g_t), \quad g\in\Loop G,
\end{equation}
with $\chi(g,t)\in \Uone $. Note that
\begin{equation}
\chi(g,2\pi)\unit_{\H_\lambda} = U^\lambda(2\pi)\hat{\lambda}(g)U^\lambda(2\pi)^*\hat{\lambda}(g)^* 
\end{equation}
so that $g \mapsto \chi(g,2\pi)$ is a continuous character of $\Loop G$ and hence  $\chi(g,2\pi)=1$ for all $g\in \Loop G$ because $\Loop G$ is a perfect group 
\cite[3.4.1]{PS}. Hence, if $\lambda$ is irreducible, then $U^\lambda(2\pi) = e^{2\pi i L_0^\lambda}$ is a multiple of the identity and hence 
$t \mapsto U^\lambda(t)$ factors to a representation of $\mathrm{Rot}(S^1)$, the group of rotations of $S^1$.

If $\pi$ is an irreducible unitary highest weight representation of $\hat{\lg}_\C$, cf. \cite{Kac}, then it exponentiates to an irreducible projective unitary positive-energy representation $\lambda_\pi$ of $\Loop G$ \cite[Thm.6.7]{GW}, \cite[Thm.6.1.2]{Tol}. 
Moreover, there is a canonical choice of the positive operator $L_0^{\lambda_\pi}$.
The value $\ell:=\pi(c_\lg)\in\N$ is called the \emph{level} of the representation $\lambda_\pi$. In this paper we shall always deal with direct sums of such representations $\lambda_\pi$ at fixed level.
It has been shown by A. Wassermann that these are exactly the irreducible strongly continuous projective unitary positive-energy representations $\lambda$ of $\Loop G$ such that the unitary  one-parameter group $U^\lambda$ implementing the rotations action on $\Loop G$ is diagonalizable with finite-dimensional eigenspaces, see \cite[Thm. 4.1]{Tol3}. 
Equivalently they are the irreducible smooth strongly continuous projective unitary positive-energy representations of $\Loop G$ \cite[Sec. 9.3]{PS}.
Any of these representations determines a central extension of $\Loop G$ by $\Uone$ \cite{Tol3} and the equivalence class of the corresponding circle bundle on $\Loop G$ only depends on the level \cite{PS}. Accordingly, the representations of fixed level $\ell$ correspond to the representations 
determining the associated central extension. 

The \emph{Verlinde fusion ring} $R^{\ell}(\Loop G)$ is the ring of formal differences (Grothendieck ring) associated with the semiring of equivalence classes of finite direct sums of such representations $\lambda_\pi$ at level $\ell$ \cite{Ver,FHT2011a,FHT2011b,FHT2013,Douglas}. The product is the so called \emph{fusion product}. For any $\ell$, there are only finitely many (say $N$) equivalence classes of irreducible unitary highest weight level $\ell$ representations of 
$\hat{\lg}_\C$. They are all described and classified in \cite{Kac}, and the corresponding loop group representations are discussed in \cite[Sect.9]{PS}, \cf also \cite[Sect.6]{Tol}.

Given now a fixed level $\ell$, let $(\lambda_i,\H_{\lambda_i})$, with $i=0,\ldots,N-1$, denote an arbitrary fixed maximal family of mutually inequivalent irreducible representations as above at level $\ell$, with $\lambda_0$ the \emph{vacuum representation}, \ie the representation corresponding to integral highest weight $0$. Then every element in $R^{\ell}(\Loop G)$ can be written as 
\begin{equation}
\sum_{i=0}^{N-1} m_i [\lambda_i],\quad m_i \in \Z .
\end{equation} 
The fusion product is expressed in terms of the \emph{Verlinde fusion rule coefficients} $\CN_{ij}^k\in\NN$ such that
\begin{equation}\label{eq:Verlinde}
[\lambda_i]\cdot[\lambda_j] = \sum_{k=0}^{N-1} \CN_{ij}^k [\lambda_k],
\end{equation}
for all $i,j=0,\ldots,N-1$. 

For every $i=0,\dots,N-1$, one can consider the conjugate representation 
$\lambda_{\bar{i}}$ of $\lambda_i$, which is uniquely determined by the condition
\begin{equation}
\label{EqFusion1}
\CN_{ij}^0= \delta_{\bar{i}, j},\quad j=0,\dots,N-1 .
\end{equation}
The fusion rule coefficients also satisfy the identities
\begin{equation}
\label{EqFusion2}
\CN_{ij}^k=\CN_{ji}^k
\end{equation}
and
\begin{equation}
\label{EqFusion3}
\CN_{ij}^{\bar{k}}=\CN_{jk}^{\bar{i}}=\CN_{\bar{k}\bar{j}}^{i} .
\end{equation}

\medskip

In order to work with loop groups, we will need the setting of conformal nets, for which we provide some basics here.
Conformal nets describe chiral conformal CFTs in the operator algebraic approach to quantum field theory \cite{Haag}. They can be considered as a functional analytic analogue of vertex algebras \cite{FLM,Kac1998}, see \cite{CKLW}.

Let $\I$ denote the set of nondense nonempty open intervals in $S^1$ and $\Diff$ the infinite-dimensional Lie group of orientation-preserving smooth diffeomorphisms of $S^1$ \cite{Milnor}. $\Diff$ contains the group of M\"{o}bius transformations of $S^1$ which is isomorphic to $\PSL$. Accordingly we will consider $\PSL$ as a (Lie) subgroup of $\Diff$.

For $I\in\I$ we denote by $I'$ the interior of the complement of $I$, which lies again in $\I$. A \emph{local conformal net} $\A$ over $S^1$ (\cf \cite{FG,GL2}) consists of a family of von Neumann algebras $(\A(I))_{I\in\I}$
acting on a common separable Hilbert space $\H$ together with a given strongly 
continuous unitary representation $U$ of $\PSL$ on $\H$ satisfying 
\begin{itemize}
\item {\it isotony}: $\A(I_1)\subset \A(I_2)$ if $I_1\subset I_2$, for $I_1,I_2\in\I$;
\item {\it locality}: elements of $\A(I_1)$ commute with those of 
$\A(I_2)$ whenever $I_1\cap I_2 = \emptyset$;
\item {\it covariance}: $U(\gamma)\A(I)U(\gamma)^*=\A(\gamma I)$ for all $\gamma\in \PSL$
and $I\in \I$;
\item {\it positivity of the energy}:
the conformal Hamiltonian $L_0$, defined by 
the equation $U(R_\alpha)=\rme^{\rmi\alpha L_0}$ ($\forall \alpha \in \R$), 
is positive, where $(R_\alpha)_{\alpha\in\R}\subset\PSL$ stands for the rotation subgroup;
\item {\it existence, uniqueness and cyclicity of the vacuum}:
up to phase there exists a unique unit vector
$\Omega \in \H$ called the ``vacuum vector''
which is invariant under the action of $U$; moreover, 
it is cyclic for the von Neumann
algebra $\bigvee_{I\in\I}\A(I)$;
\item {\it diffeomorphism covariance}: $U$ extends (uniquely) to a 
strongly continuous projective unitary representation of $\Diff$ denoted again by $U$ and satisfying
\begin{gather*}
U(\gamma)\A(I)U(\gamma)^* = \A( \gamma I),\\
\gamma \restriction_I={\rm id}_I \Rightarrow \Ad (U(\gamma ))\restriction_{\A(I)}=\rm{id}_{\A(I)},
\end{gather*}
for all $\gamma \in \Diff$ and $I \in \I$.
\end{itemize}
There are many known important consequences of the above definition, \cf \cite{GL2} and references therein for a collection with proofs. We shall need the following three:
$\A(I)'= \A(I')$, for every $I \in \I$ ({\it Haag duality}), and $\bigvee_{I\in\I}\A(I) = B (\H)$ ({\it irreducibility}), $\A(I)$ is a factor, for every $I \in \I$ ({\it factoriality}). Note also that the separability of the vacuum Hilbert space $\H$ is now known to be a consequence of the other assumptions thanks to the recent results in \cite{MTW}.    

A conformal net is said to have the {\it trace class condition} if 
${\rm Tr}(q^{L_0})< +\infty$ for all $q \in (0,1)$.

A \emph{representation} of $\A$ is a 
family $\pi=(\pi_I)_{I\in\I}$ of (unital) *-representations $\pi_I$ of $\A(I)$ on a 
common Hilbert space $\H_\pi$ such that $\pi_{I_2}\restriction_{\A(I_1)}=\pi_{I_1}$ 
whenever $I_1\subset I_2$. The representation $\pi$ is called 
{\it locally normal} if $\pi_I$ is normal for every $I\in \I$; this is always the case if $\H_\pi$ is separable. Conversely, if $\pi$ is a cyclic locally normal representation of $\A$ then $\H_\pi$ is separable.

A locally normal representation $\pi$ of the conformal net $\A$ is always (M\"{o}bius) covariant with positive energy in the sense that there exists a unique 
strongly continuous unitary representation $U_\pi$ of the universal covering $\PSI$ such that
\begin{equation}
U_\pi(\gamma)\pi_I(a)U_\pi(\gamma)^* = \pi_{\dot{\gamma}I}(U(\dot{\gamma})aU(\dot{\gamma})^*), \quad \gamma \in \PSI \; a \in \A(I),
\end{equation}
where $\dot{\gamma}$ is the image of $\gamma$ in $\PSL$ under the covering map,  and such that 
$U_\pi(\gamma)\in \bigvee_{I\in\I}\pi_I(\A(I))$, for all $\gamma\in\PSI$, \cf \cite{DFK}. Moreover, the infinitesimal generator $L_0^\pi$ of the lifting of the rotation subgroup turns out to be positive \cite{Wei06}.

The \emph{vacuum representation}  $\pi_0$ on the vacuum Hilbert space $\H_{\pi_0}  := \H$ is defined by 
$\pi_{0 , I}(x) = x$, for all $I\in \I$ and all $x\in \A(I)$, and it is obviously locally normal. The unitary equivalence class of a locally normal representation $\pi$ up to unitary equivalence is called a {\it (DHR) sector} and denoted by $[\pi]$. If $\pi$ is irreducible then $[\pi$] is said to be an irreducible sector.

A covariant representation $\pi$ on the vacuum Hilbert space $\H$ is said to be {\it localized}  in $I_0$ if $\pi_{I'} = \id$, for all $I\supset \bar{I}_0$.  
In this case, we have $\pi_I(\A(I)) \subset \A(I)$, for all $I\in \I$ containing $I_0$,   i.e. $\pi_I$ is an endomorphism of $\A(I)$. 
If $\pi$ is any representation of $\A$ on a separable Hilbert space $\H_\pi$ and $I_0$ is any interval in $\I$ then there exists a covariant representation localized in $I_0$ and unitarily equivalent to $\pi$. 

The \emph{universal C*-algebra} \cite[Sect.5.3]{FRS} of $\A$ is the unique (up to isomorphism) unital C*-algebra $C^*(\A)$ such that
\begin{itemize}
\item[-] for every $I\in\I$, there are unital embeddings $\iota_I:\A(I)\ra C^*(\A)$, such that $\iota_{I_2}\restriction_{\A(I_1)} = \iota_{I_1}$ if $I_1\subset I_2 $, and all $\iota_I(\A(I))\subset C^*(\A)$ together generate $C^*(\A)$ as C*-algebra;
\item[-] for every representation $\pi$ of $\A$ on $\H_\pi$, there is a unique representation $\check{\pi}:C^*(\A)\ra B(\H_\pi)$ such that
\[
\pi_I = \check{\pi}\circ \iota_I,\quad I\in \I, 
\] 
\end{itemize}
\cf also \cite{CCHW,Fre90,GL1}. If $\pi$ is locally normal then we say that $\check{\pi}$ is locally normal. If $\pi$ is localized in $I_0$ then we say that 
$\check{\pi}$ is localized in $I_0$.
In the following, we shall denote $\check{\pi}$ simply by $\pi$ since it will be clear from the context whether $\pi$ is a representation of $\A$ or the corresponding representation of $C^*(\A)$. 

There is a natural correspondence between localized representations and covariant localized endomorphisms of $C^*(\A)$. 
Namely, for every locally normal representation $\pi$ of $C^*(\A)$ localized in $I_0$, there is an endomorphism $\rho$ of $C^*(\A)$
which is, in an appropriate sense, covariant and localized in $I_0$, such that $\pi = \pi_0\circ \rho$ \cite{FRS,GL1}, \cf also \cite[Sect.2]{CCHW}. 

If $\pi$ is a representation of $\A$ on a separable Hilbert space then, for any $I\in \I$,  $\pi_{I'}(\A(I'))'$ and $\pi_I(\A(I))$ are factors and, as a consequence of locality,  $\pi_I(\A(I)) \subset \pi_{I'}(\A(I'))'$. The square of the index $[\pi_{I'}(\A(I'))':\pi_I(\A(I))] \in [1,+\infty]$ of the subfactor 
$\pi_I(\A(I)) \subset \pi_{I'}(\A(I'))'$ does not depend on $I$. Its square root is called the \emph{statistical dimension} of $\pi$, it is denoted by $d(\pi)$ and depends only on the sector $[\pi]$.

If $\rho_1$ and $\rho_2$ are covariant localized endomorphisms of $C^*(\A)$ localized in the same interval 
$I_0 \in \I$, then the composition $\rho_1 \rho_2$ is again a covariant endomorphism localized in $I_0$. The equivalence class $[\pi_0 \circ \rho_1 \rho_2]$ depends only on $[\pi_0 \circ \rho_1 ]$ and $[\pi_0 \circ \rho_2]$.  As a consequence, with the operations  
\[
[\pi_0 \circ \rho_1 ] [\pi_0 \circ \rho_2]  := [\pi_0 \circ \rho_1 \rho_2] ,\quad 
 [\pi_0 \circ \rho_1] + [\pi_0 \circ \rho_2]  := [\pi_0 \circ \rho_1 \oplus \pi_0 \circ \rho_2]
 \]
 the set of equivalence classes of locally normal representations of $\A$ with finite statistical dimension becomes a commutative unital semiring (without $0$)  $\Ring$ called the \emph{DHR fusion semiring} \cite{Reh}. 
 There is a conjugation $[\pi] \mapsto \overline{[\pi]}$ in  $\Ring$ determined by $\overline{[\pi_0 \circ \rho]} = \pi_0 \circ \bar{\rho}$ where $\bar{\rho}$ is the covariant localized endomorphism conjugate to $\rho$, see \cite[Subsec.2.3.]{GL2}. $\bar{\rho}$ is defined up to unitary equivalence in $C^*(\A)$.  If $\pi_0 \circ \rho$ is irreducible then $\bar{\rho}$ is determined up to equivalence by the irreducibility of  $\pi_0 \circ \bar{\rho}$ and the fact that $\pi_0$ is equivalent to a subrepresentation of  $\pi_0 \circ \rho \bar{\rho}$.

 As in \cite{CCHW} we denote by $\RRing$ the corresponding ring of formal differences (Grothendieck ring) and call it the  \emph{DHR fusion ring} of $\A$ with the corresponding \emph{DHR fusion rules}.
 
An important class of nets is the class of {\it completely rational} nets introduced in \cite{KLM}. Complete rationality appears to be the right notion for 
rational chiral CFTs in the operator algebraic setting of conformal nets. By \cite{LX} and \cite{MTW} a net $\A$ is completely rational if and only if it has only finitely many irreducible sectors, all with finite statistical dimension. If $\A$ is completely rational then $\RRing$ is finitely generated.
We shall henceforth restrict ourselves to the case where $\A$ is completely rational.

In \cite[Sect.3]{CCHW} a locally normal universal C*-algebra is constructed for general $\A$ which is more manageable than $C^*(\A)$ and which can be used as a substitute for it if only locally normal representations are considered, as typically the case. 
As $\A$ is assumed here to be completely rational with $N<\infty$ irreducible sectors, we may however work with a simpler C*-algebra isomorphic to the locally normal C*-algebra and also defined in \cite{CCHW}. It is the \emph{reduced locally normal universal C*-algebra} 
$C^*_{\red}(\A)$, defined as follows. For each sector of $\A$ consider a fixed representative representation of $C^*(\A)$, resulting in a family of mutually inequivalent irreducible locally normal representations $\{\pi_0,...\pi_{N-1}\}$ of $C^*(\A)$, where $\pi_0$ denotes the vacuum representation. Moreover, for any $i=0,1,\dots,N-1$ let $\rho_i$ be a 
localized covariant endomorphism of $C^*(\A)$ such that $[\pi_i]=[\pi_0\circ\rho_i]$. Let $(\pi_{\red},\H_{\red})$ denote the direct sum $\bigoplus_{i=0}^{N-1}\pi_i$ on the Hilbert space $\H_{\red} := \bigoplus_{i=0}^{N-1}\H_{\pi_i}$
and $C^*_{\red}(\A):=\pi_{\red}(C^*(\A))$, which turns out to be isomorphic to 
$\bigoplus_{i=0}^{N-1}B(\H_{\pi_i}) \simeq  B(\H)^{\oplus N}$ \cite[Thm.3.3]{CCHW}. Then every locally normal representation $\pi$ of $C^*(\A)$ is quasi-equivalent to a subrepresentation of $\pi_{\red}$ and hence gives rise to a unique normal representation $\pi'$ of $C^*_{\red}(\A)$ such that $\pi'\circ\pi_{\red} = \pi$. Moreover, every normal representation of $C^*_{\red}(\A)$ arises in this way and the map $\pi \mapsto \pi'$ gives rise to an isomorphism from the category of locally normal representations of $\A$ onto the category of normal representations on $C^*_{\red}(\A)$. 
As  a  consequence, for every covariant localized endomorphism $\rho$ of $C^*(\A)$ there exists a unique endomorphism 
$\hat{\rho}$ of $C^*_{\red}(\A)$ such that  $\hat{\rho}(\pi_{\red}(x)) = \pi_{\red} (\rho(x))$ for all $x \in C^*(\A)$. 
In the following we shall often denote $\pi'$ simply by $\pi$ whenever it will be clear from the context whether $\pi$ is a representation of $\A$, the corresponding representation of $C^*(\A)$ or the corresponding representation of $C^*_{\red}(\A)$.

Another universal C*-algebra describing the locally normal representation theory of the completely rational net $\A$ is the 
\emph{compact universal C*-algebra} 
\[
\kA:= C^*_{\red}(\A)\cap K(\H_{\red}) = \bigoplus_{i=0}^{N-1}K(\H_{\pi_i}),
\]
which has been also introduced in \cite{CCHW} and will play a central role in the following. Here, as usual, $K(\H)$ denotes the C*-algebra of compact operators on the Hilbert space $\H$.   The algebra $\kA$ can be considered as a universal C*-algebra for the locally normal representation theory of  the net $\A$ because of the following  property: the restriction of unital normal representations of 
$C^*_{\red}(\A)$ to the C*-subalgebra $\kA$ gives rise to an isomorphism from the category of locally normal representations of $\A$ onto the category of nondegenerate representations of $\kA$, see \cite[Prop.3.4]{CCHW}. $\kA$ is probably the simplest C*-subalgebra of 
$C^*_{\red}(\A)$ with the above property and for $N=1$ it is the only one as a consequence of  \cite[Thm. 4]{Rosenberg1953}. On the other hand, for $N>1$ there are other C*-subalgebras with this property \cite{BL1974}. However, as shown again in \cite[Prop.3.4]{CCHW} 
$\kA$ has another property which plays a crucial role in \cite{CCH, CCHW} and will play an important role in the following namely, if $\rho$ is a covariant localized endomorphism of $C^*(\A)$ and $\pi_0\circ \rho$ has finite statistical dimension then $\hat{\rho}(\kA) \subset \kA$. 

Actually, as shown by the following proposition, these two properties completely determine $\kA$ among the C*-subalgebras of $C^*_{\red}(\A)$ so that the choice of $\kA$ is canonical.  

\begin{proposition}\label{prop.canonicalkA} Let $A$ be a C*-subalgebra of  $C^*_{\red}(\A)$ such that the restriction of unital normal representations of $C^*_{\red}(\A)$ to the C*-subalgebra $A$ gives rise to an isomorphism from the category of unital normal representations of  
$C^*_{\red}(\A)$ onto the category of nondegenerate representations of $A$. Moreover, assume that $\hat{\rho}(A) \subset A$ for every localized covariant endomorphism $\rho$ of $C^*(\A)$ such that  $\pi_0\circ \rho$ has finite statistical dimension. Then $A=\kA$. 
\end{proposition}
\begin{proof} Let $\pi$ be a representation of $A$ on a Hilbert space $\H_\pi$. Then $\pi$ is a direct sum $\pi_a\oplus \pi_b$ on 
$\H_a \oplus \H_b$ where $\pi_a$ is a nondegenerate representation on $\H_a$ and $\pi_b$ is the zero representation on $\H_b$. By assumption $\pi_a(A)'$ coincide with the  commutant of a unital normal representation  of $C^*_{\red}(\A)$ on $\H_a$  which is a type I von Neumann algebra. 
Accordingly $\pi_a(A)''$ is a type I von Neumann algebra and hence $\pi(A)''= \pi_a(A)''\oplus \C \unit_b$ is a type I von Neumann algebra. 
Since $\pi$ was arbitrary it follows that $A$ is a type I C*-algebra namely $\pi(A)''$ is a type I von Neumann algebra for all representations $\pi$ of $A$, see \cite[Sec.5.5]{Dix82}, \cite[Sec.4.6]{Sakai1971}. It follows from \cite[Thm.4.6.4]{Sakai1971} and 
\cite[Cor.4.1.10 ]{Dix82} that $\pi_i(A) \supset K(\H_{\pi_i})$, $i=0,\dots,N-1$.  Now, let $I_i := \ker(\pi_i\restriction_{A})$,  $i=0,\dots,N-1$
be the primitive ideals of $A$. Given $i \in \{0,\dots,N-1\}$ there must exist $j \in \{0,\dots,N-1\}$ such that $I_i \subset I_j$  and $I_j$ is a maximal primitive ideal of $A$ and hence a maximal proper closed two-sided ideal of $A$. Since $J_j:=\{x \in A: \pi_j(x) \in K(\H_{\pi_j})  \}$ is a closed two-sided ideal of $A$ properly containing $I_j$ it follows that $J_j=A$ and hence that $\pi_j(A) = K(\H_{\pi_j})$. Let $\rho_{\bar{j}}$ be a conjugate endomorphism for $\rho_j$ so that 
$\pi_j \circ \hat{\rho}_{\bar{j}} \hat{\rho}_i \simeq \pi_0 \circ\hat{\rho}_j\hat{\rho}_{\bar{j}} \hat{\rho}_i$ contains a subrepresentation equivalent to $\pi_i$. Since  
$\pi_j \circ \hat{\rho}_{\bar{j}} \hat{\rho}_i(A)\subset \pi_j(A) = K(\H_{\pi_j})$ it follows that $\pi_i (A) = K(\H_{\pi_i})$. Thus, since $i \in \{0,\dots,N-1\}$ was arbitrary, $A\subset \kA$. 

By assumption $A' = C^*_{\red}(\A)'$ so that  $A'' = C^*_{\red}(\A)$ and hence $A$ is weakly dense in 
$C^*_{\red}(\A)$ because $A$ is nondegenerate on $\H_{\red}$.  
For any $i=0,\dots, N-1$ let $E_i$ be the orthogonal projection of $\H_{\red}$ onto $\H_{\pi_i}$. Then, $E_iAE_i = \pi_i(A) = K(\H_{\pi_i})$,
$i=0,\dots,N-1$. For every $i \in \{0,\dots,N-1\}$ we can find a bounded sequence $x^i_n \in A$ strongly convergent to $E_i$. 
Recalling that $A\subset \kA$ we see that $x^i_n x$ converges in norm to $E_ix$ for all $x\in A$ and all $i \in \{0,\dots,N-1\}$. 
Thus $E_iAE_i = E_i\kA E_i \subset A$ and consequently  $\kA \subset A$ .
\end{proof}

For any $i=0,\dots,N-1$, we choose a lowest energy unit vector $\Omega_i$ in $\H_{\pi_i}$ and denote by $q_{\Omega_i}$ the corresponding one-dimensional projection in $K(\H_{\pi_i})$ and let $p_i$ be the unique minimal projection in $\kA$ such that $\pi_i(p_i)=q_{\Omega_i}$. 
By \cite{CCH,CCHW}, the maps $[\pi_i] \mapsto [p_i] \in K_0(\kA)$ and 
$[\pi_i] \mapsto \{\hat{\rho}_i \restriction_{\kA} \} \in KK(\kA,\kA)$ give rise to a 
surjective group isomorphism 
$$\phi^\A_{-1} : {\RRing} \to K_0 (\kA)$$ 
and to an injective ring homomorphism
\begin{equation}\label{eq:phi0A}
\phi^\A_{0} : {\RRing} \to KK(\kA,\kA)
\end{equation}
in a natural way.
It is shown in (4) in the proof of \cite[Thm.4.4]{CCHW} that $[p_i]= [ \hat{\bar{\rho}}_i ( p_0)]$, for all $i=0,\ldots N-1$.
It then follows from \cite[Thm.3.1]{CCH}, that $\phi^\A_{-1}$ and $\phi^\A_{0}$ are related through the Kasparov product by 
\begin{equation}\label{eq:phi-1A1}
\phi^\A_{-1}([\pi_0\circ\rho_i]) = [p_i] = [\hat{\bar{\rho}}_i(p_0)] = (\hat{\bar{\rho}}_i)_*([p_0]) = [p_0] \times \phi_0^\A([\pi_0\circ\bar{\rho}_i]),
\end{equation}
hence
\begin{equation}\label{eq:phi-1A2}
\phi^\A_{-1}(x) = [p_0] \times \phi_0^\A(\bar{x})  , \quad x\in \RRing.
\end{equation}

\medskip

Let us now apply the general theory of conformal nets to loop groups. Let $G$ be a connected simply-connected compact simple Lie group, with the same notation as introduced at the beginning of this section. The \emph{loop group net of $G$ at level $\ell$} on $\H:=\H_{\lambda_0}$ is defined as
\begin{equation}
\label{Eq:A_{G_{ell}}net}
\A_{G_{\ell}}(I):= \{ \lambda_0(g) : g \in \Loop_I G \}'', \quad I\in\I,
\end{equation}
where the local subgroups $\Loop_I G \subset \Loop G$, $I\in \I$, are defined by 
\begin{equation}
\label{Eq:localsubgroups}
\Loop_I G := \{g \in \Loop G:  g\restriction_{I'} = 1 \},
\end{equation}
\cf \cite{FG}. It is a conformal net. 
The locality property of the net can be proved in various ways. In particular, it follows from the following lemma that we will also use later, \cf \cite[Prop.1.1.2]{Tol2}. The fact that $G$ is simply connected is crucial here. 

\begin{lemma}
\label{Lemma:localcommutation}
 Let $\lambda: \Loop G \to U(\H_\lambda)/\Uone$ be a strongly continuous projective unitary representation of 
$\Loop G$, let $g \mapsto \hat{\lambda}(g) \in U(\H_\lambda)$ be any given choice of the representatives of $\lambda(g)$ and let $I\in\I$. 
Then every unitary operator $u_I$ satisfying $u_I\lambda(g)u_I^*=\lambda(g)$, for all $g \in \Loop_I G$, commutes with $\hat{\lambda}(g)$, 
for all $g\in \Loop_I G$.
\end{lemma}
\begin{proof}
The projective equality $u_I\lambda(g)u_I^*=\lambda(g)$, $g \in \Loop_I G$ implies that the map 
$\Loop_I G  \ni g \mapsto \chi(g) \in \Uone$ defined by 
\begin{equation}
\chi(g)\unit_{\H_\lambda} := u_I\hat{\lambda}(g)u_I^*\hat{\lambda}(g)^*,\quad g\in \Loop_I G,
\end{equation}
is a continuous character of $\Loop_I G$. But $\Loop_I G$ is a perfect group by \cite[Lemma 1.1.1]{Tol2} and hence $\chi(g)=1$ for all 
$g\in \Loop_I G$.
\end{proof}

Let $\lambda_i$, $i=0,\dots N-1$, be the irreducible level $\ell$ positive-energy representations of $\Loop G$ defined at the beginning of this section.
These representations are mutually locally unitarily equivalent, \cf \cite[IV.6]{FG}, \cite[IV.2.4.1]{Tol2} and \cite[p.12]{Was2}. Accordingly, 
for any $\lambda_i$ and any $I\in\I$ there is a unitary $u_I:\H_{\lambda_0}\ra \H_{\lambda_i}$ such that $\lambda_i(g)= u_I \lambda_0(g) u_I^*$ 
for all $g \in \Loop_I G$.  Then, we can define an irreducible representation $\pi_{\lambda_i}$ of $\A_{G_{\ell}}$ by 
\begin{equation}\label{eq:pilambdadef}
\pi_{\lambda_i,I}(x) := u_I x u_I^*, \quad x\in \A_{G_{\ell}}(I), I\in\I .
\end{equation}
Note that it follows from Lemma \ref{Lemma:localcommutation} that, for any $I\in \I$, $\pi_{\lambda_i,I}$ does not depend on the 
choice of $u_I$ so that  $\pi_{\lambda_i, I_2}\restriction_{\A_{G_{\ell}}(I_1)}=\pi_{\lambda_i, I_1}$ 
whenever $I_1\subset I_2$.

\emph{Henceforth, we make the following standing assumption on $\A_{G_{\ell}}$}:
\begin{assumption}\label{CFT-assumption}
$\A_{G_{\ell}}$ is completely rational and there exists a (necessarily unique) 
ring isomorphism $\psi_{G_{\ell}}$ of $R^{\ell}(\Loop G)$ onto $\tilde{\mathcal{R}}_{\A_{G_{\ell}}}$ such that 
$\psi_{G_{\ell}}([\lambda_i]) = [\pi_{\lambda_i}]$, for all $i=0,\dots,N-1$.
\end{assumption}

This assumption might seem a strong restriction, but it is expected to be true in general and in all explicitly computed cases it has been proven \cite[Sect.3.2]{KL05}, see also \cite[Problem 3.32]{Kaw2015}. E.g. for $G=\operatorname{SU}(n)$ at any level $\ell$ the assumption follows from the results in \cite{KLM,Was,Xu}. Moreover for any loop group it is known that from a representation of the loop group conformal net one obtains a representation of the loop group. More precisely, every locally normal representation of $\A_{G_{\ell}}$ decomposes into a direct sum of irreducibles, and there is an injective map from the irreducibles in $\tilde{\mathcal{R}}_{\A_{G_{\ell}}}$ to the irreducibles in $R^{\ell}(\Loop G)$,  cf.~\cite{CW,Henriques2017}.

Using Assumption \ref{CFT-assumption} and setting $\pi_i=\pi_{\lambda_i}$, we can now define 
\begin{equation}\label{eq:phi0}
\phi^{G_\ell}_0 := \phi_0^{\A_{G_\ell}}\circ \psi_{G_{\ell}} : R^\ell(\Loop G) \ra KK(\mathfrak{K}_{\A_{G_\ell}},\mathfrak{K}_{\A_{G_\ell}}).
\end{equation}
It then follows from \eqref{eq:Verlinde}, \eqref{eq:phi-1A1} and \cite[Thm.3.1]{CCH} that

\begin{equation}\label{eq:KK-action}
\begin{aligned}
\ [p_j] \times \phi^{G_\ell}_0 ([\lambda_i]) =& [p_0] \times \phi^{G_\ell}_0([\lambda_{\bar{j}}]) \times \phi^{G_\ell}_0([\lambda_{i}])
= [p_0] \times \phi^{G_\ell}_0([\lambda_{\bar{j}}] \cdot [\lambda_{i}])\\
=& \sum_{k=1}^{N-1} \CN_{i,\bar{j}}^{\bar{k}} [p_0] \times \phi^{G_\ell}_0([\lambda_{\bar{k}}])
= \sum_{k=1}^{N-1} \CN_{i,\bar{j}}^{\bar{k}} [p_k],
\end{aligned}
\end{equation}
for all $i,j=1,\ldots , N-1$.

\begin{remark}\label{rem:nonsimple}
If $G$ is a connected simply connected compact Lie group which is not necessarily simple then 
$G$ is the direct product of connected  simply connected compact simple Lie groups and the results in this section and in the following Section 
\ref{sec:JLO} generalize in a straightforward way. 
In this case the level $\ell=({\ell}_1,{\ell}_2, \dots , {\ell}_n)$ consists of a level $\ell_i$ (a positive integer) for each simple factor $G_i$. 
Then the net $\A_{G_{\ell}}$ is the tensor product $\A_{G_1,{\ell}_1}\otimes\A_{G_2,{\ell}_2} \dots \otimes \A_{G_n, {\ell}_n}$ where 
$\A_{G_i , {{\ell}_i}}$, $i=1,2,\dots ,n$ denotes the net defined from the vacuum representation of the loop group $\Loop G_i$ at level $\ell_i$. If $\A_{G_i, {\ell}_i}$ satisfies Assumption  \ref{CFT-assumption} for $i=1,\dots,n$, then also $\A_{G_{\ell}}$ does. 
\end{remark}

\section{An index paring and KK-theory for loop group representations}\label{sec:JLO}

We recall from Section \ref{sec:DHR} that for our fixed level $\ell$, we write $(\lambda_i,\H_{\lambda_i})$, with $i=0,\ldots,N-1$,  for an arbitrary fixed maximal family of mutually inequivalent irreducible representations at level $\ell$ as introduced there, with $\lambda_0$ the vacuum representation. Motivated by the theory and natural structure of conformal loop group nets and Assumption \ref{CFT-assumption}, which we assume throughout this paper, we would like to work now completely at the level of loop groups. Let 
$(\lambda_{\red},\H_{\lambda_{\red}})$ denote the direct sum of all $(\lambda_i,\H_{\lambda_i})$. By construction, every representation in consideration is unitarily equivalent to a subrepresentation of a suitable multiple of $\lambda_{\red}$. 
Recall from Sec. \ref{sec:DHR} that for each $\lambda_i$ there is a corresponding locally normal irreducible representation 
$\pi_{\lambda_i}$ of the net $\A_{G_\ell}$ and that accordingly we can naturally identify $\H_{\lambda_i}$ with $\H_{\pi_{\lambda_i}}$ and 
$\H_{\lambda_{\red}}$ with $\H_{\red}$. We call $\lambda_{\red}$ the \emph{reduced universal representation} or simply \emph{reduced representation}. 

Now, Assumption \ref{CFT-assumption}, and the theory  of conformal nets naturally give two universal C*-algebras associated with the level $\ell$ representations of $\Loop G$ :
the \emph{reduced universal algebra} 
\[
\BG:= C^*_{\red}(\A_{G_\ell}) =  \lambda_{\red}(\Loop G)'' \simeq \bigoplus_{i=0}^{N-1} B(\H_{\lambda_i})
\]
and the \emph{compact universal algebra} 
\[
\kG:= \mathfrak{K}_{\A_{G_\ell}} = \BG \cap K(\H_{\lambda_{\red}}) .
\] 

Then, according to our previous notation, we can use the symbol $\pi_{\lambda_i}$ for the (unique normal) representation of $\BG$ such that $\lambda_i= \pi_{\lambda_i}\circ \lambda_{\red}$, and continue to use the same symbol for its restriction to $\kG$.

Given the Hilbert space $\K:=L^2(S^1,\C^{d})$ with complex conjugation operator $\gamma$, the corresponding \emph{self-dual CAR algebra} $\CAR(\K,\gamma)$ (\cite{Ara}) is the unital graded C*-algebra generated by odd $F(f)$, for $f\in\K$, such that $[F(\bar{f}),F(g)]_+=\langle f,g\rangle \unit$ and $F(f)^*=F(\gamma f) =F(\bar{f})$, in other words the C*-algebra generated by $d$ chiral free real fermionic fields. It has so-called Ramond and Neveu-Schwarz representations. Let $(\pR,\H_\pR)$ be the minimal graded Ramond representation,  \ie the unique irreducible Ramond representation if $d$ is even or the direct sum of the two inequivalent irreducible Ramond representations if $d$ is odd, and denote its grading (a selfadjoint unitary)  by $\Gamma_{\pR}$. We will write $F^{a,\pR}_n$ for $\pR(F(f))$ where $f\in\K$ is the function defined by $f(z)=x_a z^n$ and $x_a$ denotes the $a$-th canonical basis vector in $\C^d$, and $a=1,\ldots,d$, $n\in\Z$. The rotation group acts naturally on $\CAR(\K,\gamma)$, and its infinitesimal generator (also called conformal Hamiltonian) $L_0^\pR$ in the representation $\pR$ has positive discrete spectrum. The smallest eigenvalue of $L_0^\pR$ is given by $h_R:=d/16$. We write $\H_{\pR,0,+}$ for the even part (with respect to $\Gamma_{\pR}$) of the corresponding eigenspace. Let henceforth $e_R$ be the projection onto an arbitrary but fixed one-dimensional subspace of $\H_{\pR,0,+}$. For a more expanded summary about the CAR algebra and Ramond representations with the present notation, we refer to \cite[Sect.6]{CHL} and for details and proofs to \cite{Ara,Boc} together with \cite[Sect.5.3]{GBVF}.

For any $i\in \{0,\dots, N-1\}$ we fix a lowest energy unit vector $\Omega_{\lambda_i}\in\H_{\lambda_i}$ and denote by
$q_{\Omega_{\lambda_i}}$ the orthogonal projection onto $\C \Omega_{\lambda_i}$. Now, for every representation $(\lambda_i,\H_{\lambda_i})$ of $\Loop G$, we define the degenerate representation
$\hat{\pi}_{\lambda_i}: \BG \ra B(\hat{\H}_{\lambda_i})$ on the Hilbert space 
$\hat{\H}_{\lambda_i}:= \H_{\lambda_i}  \otimes \H_{\pR}$ by 
\begin{equation}\label{eq:hatpi}
\hat{\pi}_{\lambda_i}(x) := \pi_{\lambda_i}(x)  \otimes e_R , \quad x \in \BG .
\end{equation}
This way, $\hat{\H}_{\lambda_i}$ is graded by a grading operator $\hat{\Gamma}_{\lambda_i}:= \unit_{\lambda_i}\otimes\Gamma_\pR$ and $\hat{\pi}_{\lambda_i}$ is even, \ie commutes with $\Ad \hat{\Gamma}_{\lambda_i}$. 
On $\hat{\H}_{\lambda_i}$ we can define the total conformal Hamiltonian 
\begin{equation}\label{hatL}
\hat{L}^{\lambda_i}_0 := L^{\lambda_i}_0\otimes \unit_{\H_{\pR}} + 
\unit_{\H_{\lambda_i}}  \otimes L^{\pR}_0. 
\end{equation}
It satisfies 
\begin{equation}
\mathrm{e}^{\mathrm{i}\hat{L}^{\lambda_i}_0} \left( \pi_{\lambda_i}(x) \otimes e_R \right) 
\mathrm{e}^{-\mathrm{i}\hat{L}^{\lambda_i}_0} 
= 
\left(\mathrm{e}^{\mathrm{i}L^{\lambda_i}_0} \pi_{\lambda_i}(x) \mathrm{e}^{-\mathrm{i}L^{\lambda_i}_0}\right)\otimes  e_R
\end{equation}
so that it generates the rotation action in the representation $\hat{\pi}_{\lambda_i}$.

The following proposition shows that $\hat{\pi}_{\lambda_i}$ can be considered as a ``supersymmetric representation", cf. \cite{CHL,CKL}.

\begin{proposition}\label{prop:Sugawara}
For $i=0,\ldots,N-1$, there is an odd selfadjoint operator $D_{\lambda_i}$ on $\hat{\H}_{\lambda_i}$ such that 
$D_{\lambda_i}^2 = \hat{L}^{\lambda_i}_0 - \frac{c}{24} \unit_{\hat{\H}_{\lambda_i}}$ with 
$c= \frac{d}{2} + \frac{d{\ell}}{{\ell}+h^\vee}$, where $h^\vee$ is the dual Coxeter number of $\lg_\C$.  Furthermore, the spectrum of $D_{\lambda_i}$
does not contain $0$.
\end{proposition}

\begin{proof}
For the proof we will need a couple of background facts. For complete details we refer to \cite{Tol} together with \cite[Ch.9]{PS} and \cite{Kac,KT}. The notation used here has been introduced and explained in \cite[Sect.6]{CHL}.

By construction, the positive-energy representation $\lambda_i$ of $\Loop G$ at level $\ell$ comes from a representation of the affine Kac-Moody algebra $\hat{\lg}_\C$ on (the dense subspace of finite energy vectors of) $\H_{\lambda_i}$ by integration, denoted again by $\lambda_i$. Let us consider the representation 
\[
\Lambda_i:=\lambda_i\otimes\pR:\hat{\lg}_\C\oplus\CAR(\K,\gamma) \ra B(\hat{\H}_{\lambda_i}).
\]
Its generators are given by even $J_n^{a,\Lambda_i} :=J_n^{a,\lambda_i}\otimes\unit_\pR$ and odd $F_n^{a,\Lambda_i} :=\unit_{\lambda_i}\otimes F_n^{a,\pR}$, with $a=1,\ldots,d$ and $n\in\Z$ (we shall henceforth drop the ``$\otimes \unit$'' if confusion is unlikely), and they satisfy the (anti-) commutation relations
\begin{align*}
[J_m^{a,\Lambda_i},J_n^{b,\Lambda_i}] =& \sum_{c=1}^{d} \rmi f_{abc} J_{m+n}^{c,\Lambda_i} + \delta_{m+n,0} \delta_{a,b} m \ell \unit_{\hat{\H}_{\lambda_i}}, \\
[F_m^{a,\Lambda_i},F_n^{b,\Lambda_i}]_+=& \delta_{m+n,0}\delta_{a,b} \unit_{\hat{\H}_{\lambda_i}},\\
[J_m^{a,\Lambda_i},F_n^{b,\Lambda_i}] =& 0.
\end{align*}
Here $f_{abc}$ are the structure constants of $\lg_\C$ with respect to a fixed orthonormal basis $(e_a)_{a=1,\ldots,d}$ of the Lie algebra $\lg_\C$ satisfying the orthonormality condition with respect to the normalized Killing form
\begin{equation*}
-\frac{1}{2h^\vee}\tr \left(\Ad(e_a) \Ad(e_b)\right) = \delta_{a,b} \quad a, b = 1,\dots,d \,,
\end{equation*}
where $h^\vee$ is the dual Coxeter number of $\lg_\C$.  
Moreover, we define
\[
J_n'^{a,\Lambda_i}:= -\frac{\mathrm{i}}{2} \sum_{m\in\Z} \sum_{b,c=1}^{d} f_{abc} F^{b,\Lambda_i}_m F^{c,\Lambda_i}_{n-m},
\]
which has the commutation relations
\begin{align*}
[J_m'^{a,\Lambda_i},J_n'^{b,\Lambda_i}] =& \sum_{c=1}^{d} \rmi f_{abc} J_{m+n}'^{c,\Lambda_i} + \delta_{m+n,0} \delta_{a,b} m h^\vee \unit_{\hat{\H}_{\lambda_i}}, \\
[J_m'^{a,\Lambda_i},F_n^{b,\Lambda_i}] =& \sum_{c=1}^{d}  \mathrm{i} f_{abc} F_{m+n}^{c,\Lambda_i}.
\end{align*}
Then the even operators $L_n^{\Lambda_i}$ and odd $G_n^{\Lambda_i}$, with $n\in\Z$, defined through the super-Sugawara construction 
\cite{KT} (\cf also \cite[Sect.6]{CHL}, \cite[Sect. 5.9]{Kac1998} and \cite[Sect. III.13]{Was2010}) as
\begin{equation}\label{eq:Sugawara}
\begin{aligned}
G_n^{\Lambda_i} :=&  \frac{1}{\sqrt{{\ell}+h^\vee}}\sum_{a=1}^{d} \sum_{m\in\Z} : \left( J_m^{a,\Lambda_i} + \frac13 J_m'^{a,\Lambda_i} \right) F^{a,\Lambda_i}_{n-m} : \\ 
L_n^{\Lambda_i} :=& \sum_{a=1}^{d} \left(\frac{1}{2({\ell}+h^\vee)} \sum_{m\in\Z} :J^{a,\Lambda_i}_m J^{a,\Lambda_i}_{n-m}: 
-\frac12\sum_{m\in \Z} m  :F^{a,\Lambda_i}_{m}F^{a,\Lambda_i}_{n-m} :\right)
+ \frac{d}{16} \delta_{n,0}\unit_{\hat{\H}_{\lambda_i}} 
 \end{aligned}
\end{equation} 
(where $: \; :$ stands for the normally ordered product) satisfy the Ramond super-Virasoro algebra (anti-) commutation relations
\begin{equation}
\begin{aligned}\label{eq:superVir}
    [L_m^{\Lambda_i} , L_n^{\Lambda_i}] =& (m-n)L_{m+n}^{\Lambda_i} + \frac{c}{12}(m^3 - m)\de_{m+n, 0}\unit_{\hat{\H}_{\lambda_i}},\\
    [L_m^{\Lambda_i}, G_n^{\Lambda_i}] =& \Big(\frac{m}{2} - n\Big)G_{m+n}^{\Lambda_i},\\
    [G_m^{\Lambda_i}, G_n^{\Lambda_i}]_+ =& 2L_{m+n} ^{\Lambda_i}+ \frac{c}{3}\Big(m^2 - \frac14\Big)\de_{m+n,0}\unit_{\hat{\H}_{\lambda_i}},
    \end{aligned}
\end{equation}
with central charge $c=\frac{d}{2} + \frac{d{\ell}}{{\ell}+h^\vee}$ \cite[(5.8)]{KT}.
Moreover, $L^{\Lambda_i}_0 = \hat{L}^{\lambda_i}_0$
so that choosing $D_{\lambda_i}:= G^{\Lambda_i}_0$ proves the main part of our proposition. We remark that \cite[Eq. (6.4)]{CHL} contains small mistakes which do not influence the rest of that paper though; the correct version in the Ramond case is \eqref{eq:Sugawara} here.

Concerning the spectrum of $D_{\lambda_i}$, we see from Eq. (\ref{hatL}) that $L_0^{\Lambda_i}$ is bounded below by $h_R= \frac{d}{16}$. 
Hence, we find
\[
D_{\lambda_i}^2 = L_0^{\Lambda_i} - \frac{c}{24}\unit_{\hat{\H}_{\lambda_i}}
\ge \frac{d}{16} \unit_{\hat{\H}_{\lambda_i}}  - \frac{1}{48} \Big(d + \frac{2 d{\ell}}{{\ell}+h^\vee} \Big)\unit_{\hat{\H}_{\lambda_i}}
= d \Big(\frac{1}{24}- \frac{{\ell}}{24({\ell}+h^\vee)}\Big)\unit_{\hat{\H}_{\lambda_i}} > 0,
\]
as $\ell$ and $h^\vee$ are positive.
\end{proof}

As in the preceding proof, we denote the \textit{infinitesimal generator of rotations} in the representation $\hat{\pi}_{\lambda_i}$ by $L_0^{\Lambda_i}=\hat{L}_0^{\lambda_i}$; it coincides then with $D_{\lambda_i}^2$ up to an additive constant. The representations of $\hat{\lg}_\C$ we are actually interested in are $\lambda_i$, while $\Lambda_i$ are the corresponding ones of $\hat{\lg}_\C\oplus\CAR(\K,\gamma)$ needed for the super-Sugawara construction.

\begin{proposition}\label{prop:domain}
Let $\hat{\pi}_{\red}:=\bigoplus_{i=0}^{N-1}\hat{\pi}_{\lambda_i}$ and $\delta_{\red}:=\bigoplus_{i=0}^{N-1} \delta_{D_{\lambda_i}} = \delta_{D_{\red}}$, 
where $D_{\red} :=\bigoplus_{i=0}^{N-1} D_{\lambda_i}$. 
Define the \emph{compact differentiable subalgebra} of $\BG$ as
\[
\kone:= \{x\in\kG : \hat{\pi}_{\red}(x) \in \dom(\delta_{\red}), \delta_{\red}( \hat{\pi}_{\red}(x))\in K(\hat{\H}_{\red})\}.
\]
The norm
\[
\|x\|_1 = \|x\| + \|\delta_{\red}(\hat{\pi}_{\red}(x))\|_{B(\hat{\H}_{\red})}, \quad x\in\kone.
\]
is well-defined and turns $\kone$ into a Banach algebra, so that the maps $\hat{\pi}_{\lambda_i}: \kone \ra (\dom(\delta_{D_{\lambda_i}}), \|\cdot 
\|_{B(\hat{\H}_{\lambda_i})}+\|\delta_{D_{\lambda_i}}(\cdot)\|_{B(\hat{\H}_{\lambda_i})})$, $i=0,\ldots,N-1$, are continuous.
\end{proposition}

\begin{proof}
$\|\cdot\|_1$ is clearly a norm and well-defined on $\hat{\pi}_{\red}^{-1}(\dom(\delta_{\red}))$ and turns it into a Banach algebra. Moreover, $(\kG,\|\cdot\|)$ is a Banach algebra. As $\|\cdot\|_1$ is finer than $\|\cdot\|$, we see that the intersection $\kG\cap  \hat{\pi}_{\red}^{-1}(\dom(\delta_{\red}))$ is a Banach algebra, too, w.r.t. $\|\cdot\|_1$. The fact that $(K(\hat{\H}_{\red}),\|\cdot\|_{B(\hat{\H}_{\red})})$ is complete shows then that $(\kone,\|\cdot\|_1)$ is a Banach algebra.
\end{proof}

\begin{proposition}\label{prop:stably}
The Banach algebra $\kone$ is stably-unital, and its finite-rank elements are dense.
\end{proposition}

\begin{proof}
We shall construct an approximate identity $(e_n)_{n\in\N}$ of finite-rank projections in $\kone$. Once this is done, then given any $x\in\kone$, we see that $e_n x$ is finite-rank and converges to $x$, proving also the second statement.

To start with, notice that 
$L_0^{\lambda_{\red}} =\bigoplus_{i=0}^{N-1} L_0^{\lambda_{i}}$ 
has positive discrete spectrum with every eigenvalue of finite multiplicity \cite[(9.3.4)]{PS}.
The same holds for $L_0^\pR$ and hence for 
\begin{eqnarray*}
L_0^{\Lambda_{\red}} &:=& \bigoplus_{i=0}^{N-1} L_0^{\Lambda_{i}} \\
&=& L^{\lambda_{\red}}_0 \otimes \unit_{\H_{\pR}} + 
\unit_{\H_{\lambda_{\red}}} \otimes L^{\pR}_0 . 
\end{eqnarray*}
For every $n\in\N$, let $e_n$ denote the spectral projection of $L_0^{\lambda_{\red}}$ corresponding to the eigenvalues less than $n$. 
It is a finite rank projection in $\kG$.
Similarly, let $\hat{e}_n\in B(\hat{\H}_{\red})$ denote the spectral projection of $L_0^{\Lambda_{\red}}$ corresponding to the eigenvalues less than 
$n+h_R$. Then $\hat{e}_n$ is a finite rank projection commuting with $D_{\red}$ and hence $D_{\red}\hat{e}_n$ is bounded with domain 
$\hat{\H}_{\red}$. In particular $\hat{e}_n \in \dom(\delta_{\red})$ and  $\delta_{\red}(\hat{e}_n)=0$.

Moreover, $\hat{\pi}_{\red}(e_n)=( \unit_{\H_{\red}} \otimes e_R) \hat{e}_n$ is a subprojection of $\hat{e}_n$, and for every $x\in\kone$, we have $\hat{\pi}_{\red}(e_n x)=\hat{e}_n \hat{\pi}_{\red}(x)$. It follows that $\hat{\pi}_{\red}(e_n)$ is compact and in $\dom(\delta_{\red})$, with $\delta_{\red}(\hat{\pi}_{\red}(e_n))$ compact again so that $e_n \in \kone$. In general, we furthermore have $\hat{\pi}_{\red}(\kone)\subset K(\hat{\H}_{\red})$. Recalling that $\delta_{\red}(\hat{e}_n)=0$ we have, for every $x\in\kone$,
\begin{align*}
\|e_n x -x\|_1=& \|e_n x -x \| + \| \delta_{\red}(\hat{e}_n \hat{\pi}_{\red}(x)) - \delta_{\red}(\hat{\pi}_{\red}(x)) \|_{B(\hat{\H}_{\red})}\\
=& \|e_n x -x \| + \| \hat{e}_n \delta_{\red}(\hat{\pi}_{\red}(x)) - \delta_{\red}(\hat{\pi}_{\red}(x)) \|_{B(\hat{\H}_{\red})}.
\end{align*}
Now,  $e_n \to \unit_{\H_{\red}}$ and $\hat{e}_n \to \unit_{\hat{\H}_{\red}}$ as $n \to \infty$, in the strong topology. 
As a consequence, $(e_n)_{n\in\N}$ is an approximate identity for $\kG$ and $(\hat{e}_n)_{n\in\N}$ for $K(\hat{\H}_{\red})$. As $\delta_{\red}(\hat{\pi}_{\red}(x))\in K(\hat{\H}_{\red})$ by assumption, the right hand side goes to zero. Thus $(e_n)_{n\in\N}$ is an approximate identity for $\kone$.
\end{proof}

\begin{definition}\label{def:plambda}
For every $i=0,\ldots,N-1$, the \emph{lowest energy projection} $p_{\lambda_i}$ of $\lambda_i$ is defined to be the unique minimal projection in $\BG$ such that $\pi_{\lambda_i}(p_{\lambda_i}) = q_{\Omega_{\lambda_i}}$.
\end{definition}

It is simply $p_i$ in Section \ref{sec:DHR} if one chooses $\A=\A_{G_\ell}$ there.

We recall that a subalgebra $A$ of a Banach algebra $B$ is said to be \emph{closed under holomorphic functional calculus} if for every $x\in \tilde{A}$ and every function $f$ which is defined and holomorphic on a neighborhood of the spectrum $\sigma(x)$ of $x$ in the unitalization $\tilde{B}$, the element $f(x)\in \tilde{B}$ lies in $\tilde{A}$, \cf \cite[3.App.C]{Con94}. For this it is sufficient to show that $(x-\mu\unit)^{-1}\in\tilde{A}$, for every $\mu\not\in\sigma(x)$, owing to the Cauchy integral formula.

\begin{proposition}\label{prop:kinf}
The following holds:
\begin{itemize}
\item[(i)] $p_{\lambda_i}\in\kone$, for every $i=0,\ldots,N-1$, determining a class $[p_{\lambda_i}]\in K_0(\kone)$.
\item[(ii)]  $\kone\subset \kG$ is closed under holomorphic functional calculus and $\overline{\kone}^{\|\cdot\|} = \kG$.
\item[(iii)] $K_0(\kone)=K_0(\kG) = \Z^N$.
\end{itemize}
\end{proposition}
\begin{proof}
(i) We see from the proof of Proposition \ref{prop:stably} that the spectral projections $e_n$ defined there are such that 
$D_{\red} \hat{\pi}_{\red}(e_n)$ is bounded with domain $\hat{\H}_{\red}$ for every positive integer $n$. As a consequence, for every 
$x\in \BG$, $\hat{\pi}_{\red}(e_nxe_n)$ is in the domain of $\delta_{\red}$ and  $\delta_{\red}\big(\hat{\pi}_{\red}(e_nxe_n)\big)$ has finite rank so that $e_nxe_n \in \kone$. For $n$ sufficiently large we have that $p_{\lambda_i} = e_n p_{\lambda_i} e_n$, thus 
$p_{\lambda_i} \in \kone$.

(ii) Given $x\in\tkone$ and $\mu\not\in\sigma(x)$, where the spectrum is w.r.t. the C*-algebra $\tkG$, we have to show that 
$(x-\mu\unit)^{-1}\in \tkone$. To this end, notice that $(x-\mu\unit)^{-1}\in\tkG\cap \hat{\pi}_{\red}^{-1}(\dom(\delta_{\red}))$ because both $\kG$ and $\hat{\pi}_{\red}^{-1}(\dom(\delta_{\red}))$ are closed under holomorphic functional calculus, \cf \cite[Prop.3.2.29]{BR} for the second case, while for $\kG$ this is clear since it is a Banach algebra w.r.t.~$\|\cdot\|$. Moreover, the latter reference shows
\begin{align*}
\delta_{\red}\big(\hat{\pi}_{\red}((x-\mu\unit)^{-1})\big)
=& \delta_{\red}\big((\hat{\pi}_{\red}(x)-\mu\unit)^{-1}\big)\\ 
=& - \big(\hat{\pi}_{\red}(x)-\mu\unit \big)^{-1} \delta_{\red}(\hat{\pi}_{\red}(x)) \big(\hat{\pi}_{\red}(x)-\mu\unit\big)^{-1}
\end{align*}
which is compact because $\delta_{\red}(\hat{\pi}_{\red}(x))$ is compact. Hence, $(x-\mu\unit)^{-1}\in \tkone$.

To see that $\kone$ is a norm dense subalgebra of $\kG$ it is enough to note that 
$$\lim_{n\to \infty} \|e_nx e_n - x\| = 0$$ for every 
$x \in \kG$.  

(iii) follows immediately from (ii) together with \cite[3.App.C]{Con94} or \cite[5.1.2]{Bla}.
\end{proof}

\begin{theorem}\label{th:pairing}
For every $i=0,\ldots,N-1$, $(\kone,\hat{\pi}_{\lambda_i},D_{\lambda_i})$ forms an even $\theta$-summable spectral triple with nonunital $\kone$. It gives rise to the even JLO entire cyclic cocycle $\tau_{\lambda_i}\in HE^e(\kone)$. The JLO cocycle $\tau_{\lambda_i}$ pairs with $K_0(\kG)$ and
\begin{equation}\label{eq:pairing}
\tau_{\lambda_i} (p_{\lambda_j}) = \langle [\tau_{\lambda_i}],[p_{\lambda_j}]\rangle = \delta_{ij},\quad i,j=0,\ldots,N-1.
\end{equation}
\end{theorem}

\begin{proof}
It is clear that $\hat{\pi}_{\lambda_i}$ is a representation of $\kone$ with image in $\dom(\delta_{D_{\lambda_i}})$ -- in fact, in the even part of $\dom(\delta_{D_{\lambda_i}})$ because $\unit_{\H_{\lambda_i}} \otimes e_R$ is even and $B(\H_{\lambda_i})\otimes\unit_{\H_\pR}$ is even.
As explained in \cite[Sect.6]{CHL} (\cf also \cite[III.7.(iv)]{FG} using the character formulae in \cite[Sect.10]{Kac}), both $\rme^{-t L_0^{\lambda_i}}$ and $\rme^{-t L_0^{\pR}}$ on $\H_{\lambda_i}$ and $\H_{\pR}$, respectively, are trace-class, thus so is $\rme^{-t L_0^{\Lambda_i}}$ on $\hat{\H}_{\lambda_i}$, for every $t>0$, and we have an even $\theta$-summable spectral triple. It induces an even JLO cocycle, which pairs with $K_0(\kG)$ according to Theorem \ref{th:ECCJLO}$(ii)$ and Proposition \ref{prop:kinf}(ii)\&(iii). The actual values can be computed as follows: according to Definition \ref{def:plambda}, $\hat{\pi}_{\lambda_i}(p_{\lambda_j})=0$ if $i\not=j$, and
\begin{equation}\label{eq:p-def}
p:=\hat{\pi}_{\lambda_i}(p_{\lambda_i})= q_{\Omega_{\lambda_i}} \otimes e_R,
\end{equation}
which yields $pD_{\lambda_i}p=0$, $\dim (p\hat{\H}_{\lambda_i,+}) = 1$ and $\dim (p\hat{\H}_{\lambda_i,-}) = 0$, so 
\[
\tau_{\lambda_i}(p_{\lambda_i}) = \langle [\tau_{\lambda_i}],[p_{\lambda_i}]\rangle = \ind_{p\hat{\H}_{\lambda_i,+}}(p D_{\lambda_i} p) = 1 - 0 =1. 
\]
\end{proof}

The C*-algebra $\kG$ lies in the so called bootstrap class \cite[V.1.5.4]{Bl2}. Hence we can apply the universal coefficient theorem \cite[V.1.5.8]{Bl2}
which implies that $\gamma_b: KK(\kG,\kG) \to \mathrm{End}\big({K_0(\kG)}\big)$ is a surjective isomorphism (because $\Ext_\Z^1(\Z^N,\Z^N)=0$) and 
hence, using Proposition \ref{prop:kinf}(iii), $KK(\kG,\kG) \simeq \mathrm{End}\big({K_0(\kG)}\big) \simeq \mathrm{End}(\Z^N)$. It also implies that 
$\gamma_a : K^0(\kG) \to \mathrm{Hom}\big(K_0(\kG),\Z \big)$ is a surjective isomorphism so that $K^0(\kG) \simeq \Z^N$. 
It follows that $\gamma_c: KK(\kG,\kG) \to \mathrm{End}\big({K^0(\kG)}\big)$ is a surjective isomorphism, too, so that 
$KK(\kG,\kG) \simeq \mathrm{End}\big({K^0(\kG)}\big) \simeq \mathrm{End}(\Z^N)$.

For any $i=0,\dots,N-1$, one associates to the $\theta$-summable spectral triple $(\kone,\hat{\pi}_{\lambda_i},D_{\lambda_i})$ determining $\tau_{\lambda_i}$ the Fredholm module $(\hat{\H}_{\lambda_i},\hat{\pi}_{\lambda_i},\sgn D_{\lambda_i})$. Here 
$\sgn D_{\lambda_i} = D_{\lambda_i} |D_{\lambda_i}|^{-1}$ is the signature of $D_{\lambda_i}$ (recall that $0$ is not in the spectrum of 
$D_{\lambda_i}$ according to Proposition \ref{prop:Sugawara}). It gives rise to the same index map as $D_{\lambda_i}$, \cf 
\cite[Sect.\,IV.8.$\delta$]{Con94}. We write $\varepsilon_{\lambda_i}:= [\hat{\H}_{\lambda_i}, \hat{\pi}_{\lambda_i}, \sgn D_{\lambda_i}] \in K^0(\kG)$ for the corresponding K-homology class. The following proposition is included for the sake of completeness and to make things as explicit as possible for our setting.

\begin{proposition}\label{prop:KKJLO}
The classes $\varepsilon_{\lambda_i}$ and $[p_{\lambda_i}]$, with $i=0,\ldots,N-1$, generate $K^0(\kG)$ and $K_0(\kG)$, respectively. For every $i,j=1,\ldots,N$, we have
\begin{equation}\label{pivarepsilon}
[p_{\lambda_j}]\times \varepsilon_{\lambda_i} = \tau_{\lambda_i} (p_{\lambda_j}) = \delta_{i,j}.
\end{equation}
In other words, the pairing of the JLO cocycle with $K_0(\kG)$ is given by the Kasparov product between the corresponding K-homology class and $K_0(\kG)$.
\end{proposition}

\begin{proof}
The fact that $F:= \sgn D_{\lambda_i}$ has degree 1 with $F^2=\unit$ and $F^*=F=F^{-1}$ and that $\hat{\pi}_{\lambda_i}$ has support in $\hat{\H}_{\lambda_i,+}$ implies that we may write $F=\begin{pmatrix}
0 & W^* \\ W & 0
\end{pmatrix}$ 
with a certain unitary $W$. Then
\[
\varepsilon_{\lambda_i} = \Big[ \hat{\H}_{\lambda_i}, \begin{pmatrix}
\hat{\pi}_{\lambda_i,+} & 0 \\ 0 & 0
\end{pmatrix},
\begin{pmatrix}
0 & W^* \\ W & 0
\end{pmatrix} \Big]
= \Big[ \hat{\H}_{\lambda_i}, \begin{pmatrix}
W \hat{\pi}_{\lambda_i,+}(\cdot) W^* & 0 \\ 0 & 0
\end{pmatrix},
\begin{pmatrix}
0 & \unit \\ \unit & 0
\end{pmatrix} \Big],
\]
where the last equality is a consequence of unitary equivalence via the unitary
$\begin{pmatrix}
W & 0 \\ 0 & \unit
\end{pmatrix}
\in B(\hat{\H}_{\lambda_i})$. 

Now the Kasparov product of this element with $[p_{\lambda_j}]$ can be calculated using property (5) above, as the involved algebras $\C$ and $\kG$ are trivially graded. Thus
\begin{align*}
\Big[ \hat{\H}_{\kG}, \phi_{p_{\lambda_j}} \oplus 0, \begin{pmatrix}
0 & \unit \\ \unit & 0
\end{pmatrix} \Big] 
&\times
\Big[ \hat{\H}_{\lambda_i},W \hat{\pi}_{\lambda_i,+}(\cdot) W^* \oplus 0,
\begin{pmatrix}
0 & \unit \\ \unit & 0
\end{pmatrix} \Big]\\
=& \Big[ \hat{\H}_{\lambda_i}, W\hat{\pi}_{\lambda_i,+}(\phi_{p_{\lambda_j}}(\cdot)) W^* \oplus 0,
\begin{pmatrix}
0 & \unit \\ \unit & 0
\end{pmatrix} \Big] \\
=& \delta_{i,j} \Big[ \hat{\H}, \phi_{WpW^*} \oplus 0,
\begin{pmatrix}
0 & \unit \\ \unit & 0
\end{pmatrix} \Big]
\end{align*}
as product $KK(\C,\kG)\times KK(\kG,\C) \ra KK(\C,\C)$. Here we used the identification of $\hat{\H}_{\lambda_i}$ with the standard $\Z_2$-graded Hilbert space and Hilbert $\C$-module $\hat{\H}$ and the definition of $p_{\lambda_i}$ and $\hat{\pi}_{\lambda_i,+}$ in the last line, where $p$ is as in \eqref{eq:p-def}. Now $p$ and hence $WpW^*$ are rank one projections and therefore 
$\Big[ \hat{\H}, \phi_{WpW^*} \oplus 0,
\begin{pmatrix}
0 & \unit \\ \unit & 0
\end{pmatrix} \Big]$
is the generator of $KK(\C,\C)\simeq \Z$. Thus
\[
\Big[ \hat{\H}_{\kG}, \phi_{p_{\lambda_j}} \oplus 0, \begin{pmatrix}
0 & \unit \\ \unit & 0
\end{pmatrix} \Big] 
\times
\Big[ \hat{\H}_{\lambda_i},W \hat{\pi}_{\lambda_i,+}(\cdot) W^* \oplus 0,
\begin{pmatrix}
0 & \unit \\ \unit & 0
\end{pmatrix} \Big]
= \delta_{i,j}, \quad i,j=0,\dots N-1.
\]
As $K_0(\kG)\simeq\Z^n$ and $K^0(\kG)\simeq\Z^n$, we therefore see that the elements $[p_{\lambda_j}]$ and $\varepsilon_{\lambda_i}$, with $i,j=0,\ldots,N-1$, generate $K_0(\kG)$ and $K^0(\kG)$ respectively. Together with Theorem \ref{th:pairing} we get
\[
[p_{\lambda_j}]\times \varepsilon_{\lambda_i} = \delta_{i,j} = \tau_{\lambda_i} (p_{\lambda_j}).
\]
\end{proof}

Since $[p_{\lambda_i}]$ and $\varepsilon_{\lambda_i}$, with $i=0,\dots,N-1$, generate $K_0(\kG)$ and $K^0(\kG)$, respectively, the maps $[\lambda_i] \mapsto [p_{\lambda_i}]$ and $[\lambda_i] \mapsto \varepsilon_{\lambda_i}$, $i=0,\dots,N-1$,
give rise to surjective group isomorphisms 
\begin{equation*}
\phi_{-1}^{G_\ell} : R^{\ell}(\Loop G) \to K_0(\kG)
\end{equation*}
and 
\begin{equation}\label{eq:phi1}
\phi_{1}^{G_\ell} : R^{\ell}(\Loop G) \to K^0(\kG),
\end{equation}
and we have $\phi_{-1}^{G_\ell} = \phi_{-1}^{\A_{G_\ell}}\circ \psi_{G_{\ell}}$.

We also recall the injective ring homomorphism
\begin{equation*}
\phi_0^{G_\ell}   =\phi_0^{\A_{G_\ell}}\circ \psi_{G_{\ell}}: R^{\ell}(\Loop G) \to KK(\kG,\kG),
\end{equation*}
which arises naturally from the loop group conformal nets as in \eqref{eq:phi0}. It follows from \eqref{eq:KK-action} that
\begin{equation}
\label{Eqptimesphi} 
[p_{\lambda_{\bar{j}}}]\times \phi_0^{G_\ell} ([\lambda_i]) = \sum_{k=0}^{N-1} \CN_{ij}^k  \, [p_{\lambda_{\bar{k}}}] \quad i,j \in \{0,\dots,N-1\}.
\end{equation}
Together with Proposition \ref{prop:KKJLO} we find
\begin{equation}
\label{Eqdefphi_0}
\phi_0^{G_\ell} ([\lambda_i]) \times \varepsilon_{\lambda_j} = \sum_{k=0}^{N-1} \CN_{ij}^k \, \varepsilon_{\lambda_k}\quad i,j \in \{0,\dots,N-1\}.
\end{equation}
For 
\begin{equation*}
x = \sum_{i=0}^{N-1} m_i [\lambda_{i}] \in R^{\ell}(\Loop G)
\end{equation*}
we define $\bar{x} \in R^{\ell}(\Loop G)$ by 
\begin{equation*}
\bar{x} = \sum_{i=0}^{N-1} m_i [\lambda_{\bar{i}}] .
\end{equation*}
Using this notation we see that for $x,y \in R^{\ell}(\Loop G)$ we have 
\begin{equation*}
\phi_0^{G_\ell} (x) \times \phi_{1}^{G_\ell} (y) =  \phi_{1}^{G_\ell} (xy) ,\quad  \phi_{-1}^{G_\ell} (y) \times \phi_0^{G_\ell} (x) = \phi_{-1}^{G_\ell} (\bar{x} y) . 
\end{equation*} 

We summarize the above discussion in the following theorem. 

\begin{theorem}\label{th:pairing2}
There exist necessarily unique surjective group isomorphisms
\begin{align*}
\phi_{-1}^{G_\ell} : R^{\ell}(\Loop G) \to K_0(\kG)\\
 \phi_{1}^{G_\ell} : R^{\ell}(\Loop G) \to K^0(\kG)
\end{align*}
such that $ \phi_{-1}^{G_\ell} ([\lambda_i]) = [p_{\lambda_i}]$ and $ \phi_{1}^{G_\ell} ([\lambda_i]) = \varepsilon_{\lambda_i}$, $i=0,\dots,N-1$. Moreover 
there exists a unique ring homomorphism
\[
\phi_0^{G_\ell} : R^{\ell}(\Loop G) \to KK(\kG,\kG)
\]
such that $ \phi_{-1}^{G_\ell} (y) \times  \phi_0^{G_\ell} (x) =  \phi_{-1}^{G_\ell} (\bar{x}y)$, for all $x,y \in R^{\ell}(\Loop G)$. $\phi_0^{G_\ell}$ is injective and satisfies 
$\phi_0^{G_\ell} (x) \times \phi_{1}^{G_\ell} (y) =  \phi_{1}^{G_\ell} (xy)$, for all $x,y \in R^{\ell}(\Loop G)$. Moreover, 
$$
\phi_{-1}^{G_\ell} ([\lambda_k]) \times \phi_0^{G_\ell}([\lambda_i]) \times \phi_{1}^{G_\ell} ([\lambda_j]) = \CN_{i j}^k
$$
for all $i,j,k \in \{0,\dots,N-1\}$.

\end{theorem}

\begin{remark}\label{RemarkPhi}
The maps $\phi_{1}^{G_\ell} $ and $ \phi_{-1}^{G_\ell} $ can be recovered from $\phi_0^{G_\ell} $ through the identities
\begin{equation}
\phi_{1}^{G_\ell} (x) =  \phi_0^{G_\ell} (x) \times \varepsilon_{\lambda_0},\quad  \phi_{-1}^{G_\ell} (x) = [p_{\lambda_0}] \times \phi_0^{G_\ell} (\bar{x}), \quad x \in R^{\ell}(\Loop G) .
\end{equation}
\end{remark}

\begin{remark}
Every *-endomorphism $\beta$ of $\kG$ induces an endomorphism $\beta_*$ of $K_0(\kG)$ as push-forward, defined by $\beta_*([p]):= [\beta(p)]$; it also induces an endomorphism $\beta^*$ of $K^0(\kG)$ as pull-back: for a given K-homology class 
\[
\varepsilon=\Big[\hat{\H},\pi\oplus 0, \begin{pmatrix}
0 & \unit \\ \unit & 0
\end{pmatrix}\Big] \in K^0(\kG)
\]
it is defined by
\[
\beta^*(\varepsilon):=
\Big[\hat{\H},\pi\circ\beta\oplus 0, \begin{pmatrix}
0 & \unit \\ \unit & 0
\end{pmatrix}\Big] \in K^0(\kG).
\]
According to \cite[Ex.18.4.2(a)-(b)]{Bla},  if $\rho$ is a covariant localized endomorphism of $C^*(\A_{G_\ell})$ and 
$\pi_0\circ \rho$ has finite statistical dimension then
\[
[p] \times \{\hat{\rho}\restriction_{\kG}\} = (\hat{\rho}\restriction_{\kG})_*([p]), \quad [p]\in K_0(\kG),
\]
and
\[
\{\hat{\rho}\restriction_{\kG}\}\times \varepsilon = (\hat{\rho}\restriction_{\kG})^*(\varepsilon), \quad \varepsilon\in K^0(\kG).
\]
By Remark \ref{RemarkPhi} and \eqref{eq:phi-1A1}, we can therefore express $\phi_{-1}^{G_\ell} $ and $\phi_{1}^{G_\ell} $ as 
\[
 \phi_{-1}^{G_\ell} \big(\psi_{G_{\ell}}^{-1}([\pi_0\circ \rho])\big) = (\hat{\bar{\rho}}\restriction_{\kG})_*([p_{\lambda_0}]),
\quad
\phi_{1}^{G_\ell} \big(\psi_{G_{\ell}}^{-1}([\pi_0\circ \rho])\big) = (\hat{\rho}\restriction_{\kG})^*(\varepsilon_{\lambda_0}).
\]
\end{remark}

\begin{remark} 
Let $\rho$ be a covariant localized endomorphism of $C^*(\A_{G_\ell})$ such that
$\pi_0\circ \rho$ has finite statistical dimension and suppose that moreover $\hat{\rho}$ preserves the differentiable subalgebra $\kone\subset \kG$. Then it induces a pull-back endomorphism of $HE^e(\kone)$, which we denote by $(\hat{\rho}\restriction_{\kone})^*$. 
If $\rho_i$ is a covariant localized endomorphism of $C^*(\A_{G_\ell})$ such that $\pi_0\circ \rho_i$ is equivalent to $\pi_{\lambda_i}$ then  the two explicit cocycles 
$(\hat{\rho}_i\restriction_{\kone})^*\tau_{\lambda_0}$ and $\tau_{\lambda_i}$ turn out to have the same pairing with $K_0(\kG)=K_0(\kone)$ in the sense of Theorem \ref{th:pairing}, although they do not coincide. This construction was studied in \cite{CHL} in a related context though with different underlying algebras.
\end{remark}

\section{Non-simply connected compact Lie groups and other CFT models} \label{sec:nonsymplyconnected}

In this section we discuss the generalization of the results in Sections \ref{sec:DHR} and \ref{sec:JLO} to other CFT models. The strategy is the following. We first give an abstract  formulation of the results in terms of conformal nets admitting suitable supersymmetric extensions. We then show that the results apply to a large class of CFT models including lattice models, loop group models associated to non-simply connected compact Lie groups, coset models, the moonshine conformal net having the monster group $\mathbb{M}$ as automorphism group and the even shorter moonshine net having the baby monster group $\mathbb{B}$ as automorphism group. 

\subsection{Conformal nets with superconformal tensor product.}
\label{subsec:supertensor}
In this subsection we will need the notion of graded-local conformal net (also called Fermi conformal net) which is a generalization of the one of conformal net when the axiom of locality is relaxed to graded-locality (or super-locality). These are the operator algebraic analogue of vertex operator superalgebras.  We will also need to consider the special case of superconformal nets in which the conformal symmetry admits a supersymmetric extension. Basically this means that the vacuum Hilbert space of the net carries a representation of the Neveu-Schwarz super-Virasoro algebra compatible with the diffeomorphism symmetry of the net. For the precise definitions we refer the reader to \cite{CHL,CKL}, cf. also \cite{CHKLX}. 

\begin{definition}
\label{def:supertensor} 
Let $\A$ be a conformal net. We say that $\A$ admits a {\it superconformal tensor product} if there is a graded-local net $\B$ with a graded 
positive-energy Ramond representation $\pi^\B$, \cf \cite[Thm. 2.13]{CHL}, satisfying the following properties:
\smallskip

\noindent $(i)$ The graded-local conformal net $\A \otimes \B$ is superconformal in the sense of \cite[Definition 2.11]{CHL}.
\smallskip

\noindent $(ii)$ The Ramond representation $\pi^\B$ satisfies the trace-class condition i.e. $e^{-tL_0^{\pi^\B}}$ is a trace-class operator for all $t>0$.
\smallskip

\noindent We will say that $\A\otimes \B$, or more precisely the pair $\big(\A\otimes \B, \pi^\B\big)$, is a superconformal tensor product for the (local) conformal net $\A$. 
\end{definition}

\begin{remark}
\label{remark:uniquesupertensor}
Note that the superconformal tensor products for a given conformal net $\A$ are far from being unique. In fact if $\big(\A\otimes \B, \pi^\B\big)$ is a superconformal tensor product for $\A$ 
and $\mathcal{C}$ is any superconformal net with a graded Ramond representation $\pi^{\mathcal C}$ satisfying the trace class condition, then 
$\big(\A\otimes (\B \hat{\otimes} \mathcal{C}), \pi^{\B \hat{\otimes} \mathcal{C}} \big)$, where $\pi^{\B \hat{\otimes} \mathcal{C}}$ is any graded subrepresentation of $\pi^\B \hat{\otimes}\pi^{\mathcal{C}}$, is again a superconformal tensor product for $\A$. Here, $\hat{\otimes}$ denotes the graded tensor product, see e.g. \cite[Subsec.2.6]{CKL}.
In particular one could take $\mathcal{C}=\A \otimes \B$ and $\pi^{\mathcal{C}}= \pi_0 \otimes \pi^\B$, where $\pi_0$ is the vacuum representation of $\A$. Accordingly, if the net $\A$ admits a superconformal tensor product then it admits infinitely many superconformal tensor products.This should be regarded as a benefit rather than a disadvantage as it allows for greater flexibility. In particular, in Proposition \ref{prop:strictpositivitysupertensor} we will make use of this fact. As we shall discuss later in Remark \ref{RemarkNonUniqueSCTensor} two different choices of the superconformal tensor product can be both considered natural from different point of view.
\end{remark}

\begin{remark}
\label{remark:VOAsupertensor} 
The notion of superconformal tensor product can be defined in a completely analogous way for vertex operator algebras.
\end{remark}

The motivating examples for the above definition are the conformal nets $\A_{G_{\ell}}$ considered in Section \ref{sec:DHR}. Let $\F$ be the graded local conformal net on $S^1$ generated by a real free Fermi field (the {\it free Fermi net}), see e.g. \cite{CHL,CKL}. Then for every positive integer $n$, the net $\F^n$ generated 
by $n$ real free Fermi fields can be defined inductively by $\F^1:=\F$ and $\F^{n+1}:= \F^{n} \hat{\otimes} \F$, $n\in\N$.
The super-Sugawara construction described in Section \ref{sec:JLO} shows that      
 the net $\A_{G_{\ell}}$ admits a superconformal tensor product $\big(\A_{G_{\ell}} \otimes \B, \pi^\B  \big)$, where $\B =\F^{d}$ with $d$ the dimension of $G$ and $\pi^\B= \pi_R$ is the minimal graded Ramond representation, \ie the unique irreducible Ramond representation of $\F^d$ if $d$ is even or the direct sum of the two inequivalent irreducible Ramond representations of $\F^d$ if $d$ is odd. Then, in this case, the superconformal nets  $\A_{G_{\ell}}\otimes \F^d$ are the  super-current algebra nets considered in \cite[Sec.6]{CHL}.
Note that for $d$ odd the irreducible Ramond representations of $\F^d$ are not graded and for this reason we have to chose a reducible Ramond representation $\pi_R$. However, the irreducibility of $\pi^\B$ in a superconformal tensor product $\big(\A \otimes \B, \pi^\B \big)$ is not necessary for the purposes of this paper. The important property is the trace-class condition for $\pi^\B$. 

Many other examples can be given thanks to the following two propositions whose proofs are rather straightforward and will be omitted here. 

\begin{proposition}
\label{prop:extsupertensor}
Let $\A$ be a conformal net on $S^1$ with a superconformal tensor product $\big(\A\otimes \B, \pi^\B\big)$. If $\tilde{\A}$ is an irreducible local extension of $\A$ then $\big(\tilde{\A} \otimes \B, \pi^\B\big)$ is a superconformal tensor product for $\tilde{\A}$. 
\end{proposition}

\begin{proposition}
\label{prop:tensorsupertensor}
If $\A_1$, $\A_2$ are conformal nets with superconformal tensor products $\big(\A_1\otimes \B_1, \pi^{\B_1}\big)$ and 
$\big(\A_2 \otimes \B_2, \pi^{\B_2}\big)$ respectively then, for every graded subrepresentation $\pi^{\B_1\hat{\otimes} \B_2}$ of $\pi^{\B_1} \hat{\otimes} \pi^{\B_2}$, 
$\big((\A_1 \otimes \A_2)\otimes(\B_1\hat{\otimes} \B_2), \pi^{\B_1 \hat{\otimes}\B_2}\big)$ is a superconformal tensor product for the local conformal net $\A_1\otimes \A_2$ with $\pi_0$ the vacuum representation. 
\end{proposition}

Now, let $\A$ be a completely rational conformal net with $N$ irreducible sectors and let $\{\pi_0, \pi_1,\dots,\pi_{N-1}\}$ be a maximal family of irreducible locally normal representations of $\A$, with $\pi_0$ the vacuum representation. Following the general general notation of conformal nets in Section \ref{sec:DHR}, recall the definition of the surjective group isomorphism 
$$\phi^\A_{-1} : {\RRing} \to K_0 (\kA)$$ 
and the injective ring homomorphism 
$$\phi^\A_{0} : {\RRing} \to KK(\kA,\kA).$$ 
In order now to define a surjective group isomorphism $\phi^\A_1: \RRing \to K^0(\kA)$ by means of Dirac operators in analogy to the loop group setting as in \eqref{eq:phi1}, we assume that the completely rational conformal net $\A$ admits a superconformal tensor product $\big(\A\otimes \B,\pi^\B\big)$. We also assume
that $\A$ has a {\it trace-class representation theory}, namely that $e^{-tL_0^\pi}$ is trace class for all $t>0$ and all irreducible locally normal 
representations $\pi$ of $\A$.  All the completely rational conformal nets we know have a trace-class representation theory, \cf  \cite[Subsec.3.2.]{KL05}. Note that if $\A$ is a modular net in the sense of \cite{KL05} then $\A$ has a trace-class representation theory and it is conjectured that all completely rational nets have trace class representation theory \cite[Conjecture 4.18]{Kaw2015}. Now, for every locally normal representation $\pi$ of $\A$ with finite statistical dimension we consider the Ramond representation $\dot{\pi}$ of the superconformal net $\A\otimes \B$ defined by 
$\dot{\pi}:= \pi \otimes \pi^\B$. Then, by our assumptions, the conformal Hamiltonian 
$$L_0^{\dot{\pi}}=L_0^\pi \otimes \unit_{\H_{\pi^\B}} + \unit_{\H_\pi} \otimes L_0^{\pi^\B}$$ 
has non negative spectrum and $e^{-tL_0^{\dot{\pi}}}$ is trace-class for all $t>0$. Now, by the proof of \cite[Prop. 2.14]{CHL}, there is a unitary positive-energy representation of the Ramond super-Virasoro algebra on $\H_{\dot{\pi}}$ (cf. (\ref{eq:superVir})) by operators
$L_n^{\dot{\pi}}$, $G_r^{\dot{\pi}}$, $n,r \in \Z$, with central charge $c=c_\A + c_\B$, where $c_\A$ and $c_\B$ are the central charges of 
$\A$ and $\B$ respectively.  Then, the Dirac operator $D_{\pi}:=G_0^{\dot{\pi}}$ satisfies 
$D_{\pi}^2= L_0^{\dot{\pi}} -\frac{c}{24}\unit_{\H_{\dot{\pi}}} \geq (h_\B - \frac{c}{24})\unit_{\H_{\dot{\pi}}}$, where $h_\B$ is the lowest energy in the representation $\pi^\B$. Since $h_\B$ is an eigenvalue of $L_0^{\tilde{\pi_0}}$ we see that we always have $h_\B - \frac{c_\A+c_\B}{24} \geq 0$ and we say that the superconformal tensor product
$\big(\A\otimes \B,\pi^\B\big)$ for $\A$ has the {\it strict positivity property} if $h_\B - \frac{c_\A+c_\B}{24} > 0$. In this case 
the Dirac operator $D_{\pi}$ has trivial kernel for every locally normal representation $\pi$ of $\A$ with finite statistical dimension. 

\begin{proposition}  
\label{prop:strictpositivitysupertensor}
If a conformal net $\A$ admits a superconformal tensor product then it also admits a superconformal tensor product with the strict positivity property. 
\end{proposition}
\begin{proof} Let $(\A \otimes \B, \pi^\B )$ be a superconformal tensor product for $\A$ and let $\mathcal{C}:= \A_{G_{\ell}}\otimes \F^d$ be a super-current algebra net, where $G$ is a simply connected compact simple Lie group, $d$ is the dimension of $G$ and the level $\ell$ is a positive integer.   Moreover, let $\pi^{\mathcal{C}}:=\pi_0 \otimes \pi_R$ with $\pi_R$ the minimal graded Ramond representation of $\F^d$ defined above. Then $\mathcal{C}$ is a superconformal net with central charge $c_{\mathcal{C}}= \frac{d}{2} + \frac{d{\ell}}{{\ell}+h^\vee}$ and 
$\pi^{\mathcal{C}}$ is a graded Ramond representation of $\mathcal{C}$ with lowest energy 
$h_{\mathcal{C}} = \frac{d}{16}$. Accordingly 
$$h_{\mathcal{C}} -\frac{c_{\mathcal{C}}}{24} = d\left(\frac{1}{24} -  \frac{{\ell}}{24({\ell}+h^\vee)} \right)>0 .$$
Now, let $\tilde{\B} := \B \hat{\otimes} {\mathcal{C}}$ and $\pi^{\tilde{\B}} := \pi^\B \hat{\otimes} \pi^{\mathcal{C}}$. 
Then, $\tilde{\B}$ has central charge  $c_{\tilde{\B}} = c_\B + c_{\mathcal{C}}$ and $\pi^{\tilde{\B}}$ 
has lowest energy $h_\B + h_{\mathcal{C}}$ so that 
\begin{eqnarray*}
h_{\tilde{\B}} - \frac{c_\A+c_{\tilde{\B}}}{24} &=& h_\B - \frac{c_\A+c_\B}{24} + h_{\mathcal{C}} -\frac{c_{\mathcal{C}}}{24} \\
&\geq& h_{\mathcal{C}} -\frac{c_{\mathcal{C}}}{24} >0. 
\end{eqnarray*}
Hence, $(\A \otimes \tilde{\B}, \pi^{\tilde{\B}})$ is a superconformal tensor product for $\A$ with the strict positivity property. 
\end{proof}

Let $(\A \otimes \B, \pi^\B )$ be a superconformal tensor product for $\A$ having the strict positivity property. We fix a projection $e_\B \in B(\H_{\pi^\B})$ onto any one-dimensional even lowest energy subspace of $\H_{\pi^\B}$. Let $\pi$ be a locally normal representation of $\A$ with finite statistical dimension and let $\pi'$ be the unique normal representation of the reduced universal C*-algebra $C^*_{\red}(\A)$ such that $\pi'\circ \pi_{\red} = \pi$. 
We define a degenerate representation $\hat{\pi}$ of $\kA$ on $\H_{\dot{\pi}} =\H_\pi \otimes \H_{\pi^\B}$ by 
$\hat{\pi}(x) := \pi'(x)\otimes  e_\B$.   In particular we can define pairwise unitarily inequivalent representations $\hat{\pi}_i$, $i=0,\dots,N-1$. 
Using the Dirac operators $D_{\pi_i}$, we can define the Fredholm modules $({\H}_{\tilde{\pi_i}},\hat{\pi}_{i},\sgn D_{\pi_i})$ for the C*-algebra 
$\kA$ and the corresponding K-homology classes $\varepsilon_i:= \left[({\H}_{\tilde{\pi_i}},\hat{\pi}_{i},\sgn D_{\pi_i})\right] \in K^0(\kA)$. 
Then, as in Proposition \ref{prop:KKJLO} it can be shown that the Kasparov product with the K-theory classes $[p_i]$ gives 
\begin{equation}
\label{eq:piKasparovepsilon}
[p_i] \times \varepsilon_j = \delta_{i,j}, \quad i, j = 0,\dots,N-1.
\end{equation}
Consequently, the K-homology classes $\varepsilon_i$ do not depend on the choice of the superconformal tensor product 
$(\A \otimes \B, \pi^\B )$. 
Now, the maps $[\pi_i] \mapsto \varepsilon_i$ give rise to a unique group isomomorphism $\phi_1^\A: \RRing \to K^0(\kA)$.   
As a consequence of Eq. (\ref{eq:piKasparovepsilon}) and of the results in \cite{CCH,CCHW}, we have the following theorem.

\begin{theorem} 
\label{abstractfusionKtheory}
Let $\A$ be a completely rational conformal net having $N$ irreducible sectors and admitting a superconformal tensor product. 
Then the surjective group isomorphism  $\phi_1^\A: \RRing \to K^0(\kA)$ does not depend on the choice of the superconformal product. Moreover, 
\[
\phi_0^\A(x)\times \phi_1^\A(y)=\phi_1^\A(xy), \quad x,y\in\RRing.
\]
In particular,
\[
\phi_{-1}^\A(x_k) \times \phi_0^\A (x_i) \times \phi_1^\A(x_j) = \CN_{i j}^k,
\]
where $x_i:=[\pi_i]$, $i=0,\ldots,N-1$, denotes the preferred basis of $\RRing$.
\end{theorem}

\begin{remark}
\label{remark:spectraltriples}
If $\A$ is a completely rational conformal net with $N$ irreducible sectors which has trace-class representation theory and admits a superconformal tensor product then, following 
the arguments in Section \ref{sec:JLO}, one can define a differentiable Banach algebra $\kA^1 \subset \kA$ which is a dense subalgebra 
of $\kA$ closed under holomorphic functional calculus so that $K_0(\kA^1)=K_0(\kA)$. Moreover, as in Theorem \ref{th:pairing} one can define 
even $\theta$-summable spectral triples $(\kA^1, \hat{\pi}_i, D_{\pi_i})$ whose JLO cocycles $\tau_i$ gives the same index pairing as the K-homology classes $\varepsilon_i$, namely $\tau_i(p_j)= [\pi_j] \times \varepsilon_i = \delta_{i,j}$, $i,j = 0,\dots,N-1$. Accordingly the entire cyclic cohomology classes $[\tau_i]$ give complete noncommutative geometric invariants for the irreducible sectors of $\A$. However, in contrast to the surjective group isomorphism 
$\phi_1^\A:\RRing \to K^0(\kA)$ the differential algebra $\kA^1$ and hence the corresponding cyclic cohomology classes $[\tau_i]$ may depend on the choice of the superconformal tensor product for $\A$. 
\end{remark}

\begin{remark} If $\A$ is an arbitrary completely rational conformal net with sectors $[\pi_i]$, $i=0,\dots, N-1$ then one can define elements 
$\varepsilon_i \in K^0(\kA)$ determined by Eq. (\ref{eq:piKasparovepsilon}). Then the map $[\pi_i] \mapsto \varepsilon_i$ defines a group isomorphism 
 $\phi_1^\A$ from $\RRing$ onto  $K^0(\kA)$ such that $\phi_0^\A(x)\times \phi_1^\A(y)=\phi_1^\A(xy)$ as in Theorem \ref{abstractfusionKtheory}. However, without assuming that $\A$ admits a superconformal tensor product and a trace-class representation theory we loose the natural interpretation of 
$\phi_1^\A$  in terms of Dirac operators, JLO cocycles and superconformal symmetry. 
\end{remark}

\begin{remark} Let $\A$ be a completely rational conformal net. Then $K^0(\kA)$ is a finitely generated free abelian group. Every ring structure on 
$\K^0(\kA)$ obtained by introducing a product $\star$ compatible with the group operation $+$ gives rise to a group homomorphism 
$\phi_0^\star : K^0(\kA) \to KK(\kA,\kA)$ determined by the condition $x\star y = \phi_0^\star(x) \times y$, $x,y \in K^0(\kA)$.  Conversely every group homomorphism
$\phi : K^0(\kA) \to KK(\kA,\kA)$ determines a ring structure on $K^0(\kA)$ through the product $x\star_\phi y =\phi(x) \times y$. 
With the special choice $\phi:= \phi_0^\A\circ \big(\phi_1^\A \big)^{-1}$, with $\phi_0^\A$ determined by the action of the DHR endomorphisms on $\kA$, the surjective group isomorphism $\phi^\A: \RRing \to K^0(\kA)$ becomes a surjective ring isomorphism.

\end{remark}

\subsection{Applications to chiral CFT models} 
\label{subsec:CFTmodels}
In this subsection we give various examples of conformal nets admitting superconformal tensor products. For the examples for which the net is known to be completely rational we can apply Theorem \ref{abstractfusionKtheory} and Remark \ref{remark:spectraltriples} so that the DHR fusion ring of the nets can be described in terms of K-theory and noncommutative geometry. All examples considered below, completely rational or not, have a vertex operator algebra analogue.

\begin{example} 
\label{ex:Uone} 
Let $\A_{\Uone}$ be the conformal net generated by a chiral $\Uone$ current (chiral free Bose field) considered in \cite{BMT}, see also 
\cite[Example 8.6]{CKLW}. Then it follows from the super-Sugawara construction in \cite{KT}, cf. Section \ref{sec:JLO}, 
\cite[Sec.5.9]{Kac1998} and also \cite[Sec.6]{CHL}, that $(\A_{\Uone} \otimes \F, \pR)$ is a superconformal tensor product for $\A_{\Uone}$, where $\F$ is the free Fermi net and $\pi_R$ is the corresponding graded Ramond representation.  More generally, for any positive integer $n$, the conformal net $\A_{\Uone^n}$ defined as the tensor product of  $n$ copies of $\A_{\Uone}$, admits the superconformal tensor product $(\A_{\Uone^n} \otimes \F^n, \pi_R)$. All the nets $\A_{\Uone^n}$ admit uncountably many irreducible sectors, see e.g. \cite{BMT}, and hence cannot be completely rational. 
\end{example}

\begin{example}
\label{ex:lattices}
Let $L$ be an even positive-definite lattice of rank $n$. Then one can define a corresponding conformal net $\A_L$ \cite{DX2006,Staszkiewicz1995}, which is completely rational by \cite[Corollary 3.19]{DX2006}.  The conformal net $\A_L$ is the operator algebraic analogue of the simple lattice vertex operator algebra $V_L$, see \eg \cite[Sect.5.5]{Kac1998} for the definition of $V_L$, see also \cite[Conjecture 8.17]{CKLW}. By construction, the net 
$\A_L$ is an irreducible local extension of $\A_{\Uone^n}$ and hence it admits a superconformal tensor product. In fact one can choose $(\A_L\otimes \F^n,\pi_R)$. 
\end{example}

\begin{remark}
\label{remark:lattice}
The lattice net $\A_L$ and its representations are related to the projective unitary positive-energy representations
of the loop group $\Loop \Uone^n$ corresponding to a central extension determined by $L$ through the group isomorphism
between $\Uone^n$ and the $n$-dimensional torus $\R L /L$, see \cite[Sect.9.5]{PS} and \cite[Sect.3]{DX2006}. Accordingly 
$\A_L$ may be considered as a loop group net for the non-simply connected group $U(1)$. The subnet 
$\A_{\Uone^n} \subset \A_L$ then corresponds to the restriction to the component of the identity $\big(\Loop \Uone^n\big)_1$ of $\Loop \Uone^n$.

\end{remark}

\begin{example}
\label{example:non-simplyconnected}
Let $G$ be a compact connected Lie group. Then 
\begin{equation}
G \simeq \left(G_1\times G_2 \times \cdots \times G_m \times \Uone^n \right)/ Z\, , 
\end{equation}
where $G_i$, $i=1,\dots,m$ is a connected simply connected compact simple Lie group and $Z$ is a finite subgroup of the center 
$Z(G_1)\times Z(G_2) \times \cdots \times Z(G_m) \times \Uone^n $ of 
$G_1\times G_2 \times \cdots\times G_m \times \Uone^n$. Of course, $G$ is semisimple if and only if $n=0$ \ie there is no torus factor. 
Now, let $BG$ be the classifying space of $G$.  A class $\ell \in H^4(BG,\Z)$ is called a level  and transgresses to a central extension 
${\Loop G}^{\ell}$ of $\Loop G$, \cite{FHT2011b,Henriques2015,Waldorf2015}. We say that the level $\ell \in H^4(BG,\Z)$ is {\it positive} if 
${\Loop G}^{\ell}$ admits irreducible positive-energy unitary representations \ie if there exists an irreducible positive-energy projective unitary representations at level $\ell$ \cite{Tol3}, see also Definition 1 and Definition 2 in \cite{Henriques2015}. If $\ell \in H^4(BG,\Z)$ is a positive level then one can define a local conformal net $\A_{G_{\ell}}$ as a simple current extension  
\begin{equation}
\A_{G_{\ell}}:=  \left(\A_{G_1,{\ell}_1}\otimes\A_{G_2,{\ell}_2} \dots \otimes \A_{G_m, {\ell}_m} \otimes \A_L\right) \rtimes Z\, , 
\end{equation}
where the positive integers (levels) $\ell_1,\cdots , {\ell}_m$ and the lattice $L$ are determined by $\ell \in H^4(BG,\Z)$, see 
\cite{Henriques2015, Henriques2016}. Equivalently, the net $\A_{G_{\ell}}$ can be defined through the vacuum representation $\lambda_0$ of 
${\Loop G}^{\ell}$ as in the 
simply connected case, see Eq. (\ref{Eq:A_{G_{ell}}net}). The locality property of the net follows from the 
disjoint-commutativity of the
``transgressive'' central extension ${\Loop G}^{\ell}$, see \cite[Sect.3.3]{Waldorf2015}. Now, $\A_{G_{i, {\ell}_i}}$, $i=1,\dots,m$ and $\A_L$ admit 
a superconformal tensor product and hence, by Proposition \ref{prop:tensorsupertensor}, the net 
\begin{equation}
\label{Eq:Atilde}
\tilde{\A}:= \A_{G_1,{\ell}_1}\otimes\A_{G_2,{\ell}_2} \dots \otimes \A_{G_m, {\ell}_m} \otimes \A_L
\end{equation}
admits a superconformal tensor product. Accordingly, since $\A_{G_{\ell}}=\tilde{\A}\rtimes Z$ is an irreducible local extension of $\tilde{\A}$, it also admits a superconformal tensor product by Proposition \ref{prop:extsupertensor}. Actually, the supersymmetric tensor product can be taken of the form 
$\big(\A_{G_{\ell}}\otimes \F^d, \pi_R  \big)$ where $d$ is the dimension of $G$, $\F^d$ is the graded-local conformal net generated by $d$ real free Fermi fields and $\pi_R$ is a graded Ramond representation of $\F^d$. Hence, if $\A_{G_{\ell}}$ is completely rational, equivalently if $\A_{G_{i, {\ell}_i}}$ is completely rational for all $i=1,\dots,m$, Theorem \ref{abstractfusionKtheory} applies and gives a generalization of the results in Section \ref{sec:DHR} to the case of non-simply connected compact Lie groups. 
\end{example}

\begin{remark} Similarly to Example \ref{example:non-simplyconnected} one can define a vertex operator algebra $V_{G_{\ell}}$, which is the analogue 
of the conformal net $\A_{G_{\ell}}$, as a simple current extension of the tensor product of an affine vertex operator algebra and a lattice vertex operator algebra, \cf \cite{DLM1996, Henriques2015, Henriques2016, Li2001}. 
\end{remark}

\begin{remark}
Let $G$ and $\ell \in H^4(BG,\Z)$ as in Example \ref{example:non-simplyconnected} and let $R^{\ell}(\Loop G)$ be the free abelian group generated by the irreducible positive-energy unitary representations of ${\Loop G}^{\ell}$, \ie the level $\ell$ irreducible positive-energy projective unitary representations of $\Loop G$, with $e^{i2\pi L^\lambda_0}$ a multiple of the identity, see \eg \cite{FHT2011b}. Note that the condition that 
$e^{i2\pi L^\lambda_0}$ is a multiple of the identity is not necessarily satisfied for general positive-energy representations, see \cite{Tol3}.
Then $R^{\ell}(\Loop G)$ admits a fusion ring structure which can be defined through the correspondence with the VOA simple modules of the vertex operator algebra $V_{G_{\ell}}$ or through the modular invariance property of characters as in the simply connected case. Moreover, as in the simply connected case, any  irreducible positive-energy unitary representation $\lambda$ of 
${\Loop G}^{\ell}$ is locally unitarily equivalent to the vacuum representation $\lambda_0$ and hence gives rise to an irreducible locally normal representation $\pi_\lambda$ of the conformal net $\A_{G_{\ell}}$ as in Eq. (\ref{eq:pilambdadef}) and, if $\A_{G_{\ell}}$ is completely rational, the map 
$\lambda \mapsto \pi_\lambda$ gives rise to a surjective group isomorphism $\psi_{G_{\ell}}: R^{\ell}(\Loop G) \to \tilde{\mathcal{R}}_{\A_{G_{\ell}}}$. It is expected that $\A_{G_{\ell}}$ is always completely rational and that $\psi_{G_{\ell}}$ is always a surjective ring isomorphism, i.e. that $\A_{G_{\ell}}$ satisfies the analogue of 
Assumption \ref{CFT-assumption} for every connected compact Lie group $G$ and every positive level $\ell \in H^4(BG,\Z)$, see e.g. 
\cite[Conjecture 4]{Henriques2015}, although this remains in general an important open problem. A positive solution is known for some special cases, e.g. for $G=\mathrm{SO}(3)= \mathrm{SU(2)}/\Z_2$ at every positive level, \cf \cite[Sect. 3]{Henriques2015} and \cite[Sect.5]{Bischoff2015}.  

\end{remark}

\begin{remark}\label{RemarkNonUniqueSCTensor} Let $\A := \A_{\Uone_1}$ be the conformal net associated with the loop group 
$\Loop \Uone$ at level $\ell = 1 \in H^4(B\Uone,\Z)$. Note that because of a different convention the corresponding chiral CFT is often called $\Uone$ at level $2$, \cf \cite[Subsec.1.2.]{Henriques2016}. Then, $\A$ is a rank one lattice net. It is known that $\A = \A_{\mathrm{SU}(2)_1}$, see \eg \cite{BMT} and  \cite[Subsec.1.2.]{Henriques2016}. It follows that if $\F$ is the free Fermi net as in Example  \ref{ex:Uone} then we can 
define two natural superconformal tensor products for $\A$, namely $\left( \A\otimes \F, \pi_R\right)$ and 
$\left( \A\otimes \F^3, \pi_R\right)$. Note that the Dirac operators for the first choice are related to the equivariant families 
 in \cite[Part.V]{FHT2011b} for $\Loop \Uone$ while the Dirac operators for the second choice are related to the equivariant families 
for $\Loop \mathrm{SU}(2)$. The corresponding twisted K-theory classes are related by the results in \cite{FHT2011b} and the isomorphism  $R^{1}(\Loop \Uone) \simeq R^{1}(\Loop \mathrm{SU}(2))$.  On the other hand these two different choices of the superconformal tensor product and the corresponding Dirac operators give through Thm. \ref{abstractfusionKtheory} the same K-homology classes in 
$K^0(\kA) \simeq K_0(\kA) \simeq \tilde{\mathcal{R}}_{\A}$. There are many other examples of this type. This shows that if one looks at a conformal net $\A$ alone without further structure there seems to be no natural choice of the superconformal tensor product and hence of the Dirac operators. The important point for our construction is that the group isomorphism $\phi^\A_1: \tilde{\mathcal{R}}_{\A} \to K^0(\kA)$ in Thm. \ref{abstractfusionKtheory}   does not depend on this choice.
\end{remark}

\begin{example}
\label{example:coset}
Let $G$ and $\ell \in H^4(BG,\Z)$ be as in Example \ref{example:non-simplyconnected} and let $H \subset G$ be a closed connected subgroup of $G$. 
Then the positive level $\ell \in H^4(BG,\Z)$ maps to a positive level ${\ell}' \in H^4(BH,\Z)$ which gives a smooth central extension 
${\Loop H}^{{\ell}'} \subset {\Loop G}^{\ell}$ of $\Loop H$. The restriction to ${\Loop H}^{{\ell}'}$ of the vacuum representation 
$\lambda_0$ of ${\Loop G}^{\ell}$ 
gives rise to an embedding $\A_{H_{{\ell}'}} \subset \A_{G_{\ell}}$ of the conformal net $\A_{H_{{\ell}'}}$ as a covariant subnet of  
$\A_{G_{\ell}}$. 
The corresponding coset subnet $\A_{H_{{\ell}'}}^c \subset \A_{G_{\ell}}$ can be defined by the relative commutant
\begin{equation}
\label{Eq:cosetsubnet}
\A_{H_{{\ell}'}}^c(I):= \A_{H_{{\ell}'}}(I)' \cap \A_{G_{\ell}}(I)\,, \quad I\in \I\,,
\end{equation}
see \cite{Xu2000b}, see also \cite{CKLW,Longo2003}. It follows from the Kazama-Suzuki superconformal coset construction \cite{KazamaSuzuki1,KazamaSuzuki2} 
that the coset conformal net $\A_{H_{{\ell}'}}^c$ admits a superconformal tensor product 
$\big(\A_{H_{{\ell}'}}^c\otimes \F^{d_G-d_H}, \pi_R \big)$, 
where $d_G$ and $d_H$ are the dimensions of $G$ and $H$ respectively. In various cases the coset net $\A_{H_{{\ell}'}}^c$ is known to be 
completely rational, see \cite{Longo2003,Xu1999,Xu2000b,Xu2001,Xu2005}. In all these cases 
Theorem \ref{abstractfusionKtheory} applies. Moreover, the completely rational coset conformal nets admit interesting irreducible extensions such as the mirror extensions defined in \cite{Xu2007}. Then all these extensions are completely rational and admit a superconformal tensor product so that Theorem \ref{abstractfusionKtheory} applies.  Many interesting CFT models can be described through completely rational coset conformal nets and their irreducible extensions. Various examples will be given here below. 
\end{example}

\begin{example}
\label{Example:VirasoroNets}
Let $\A_{\Vir,c}$ be the Virasoro net with central charge $c$ \cite{Carpi2004,KL04}. If $\A$ is a conformal net then the corresponding representation
of $\Diff$ gives rise to an irreducible subnet $\A_{\Vir,c} \subset \A$, the Virasoro subnet of $\A$. The value $c$ is determined by $\A$ and  $c$ is said to be the central charge of $\A$. If $c<1$ then, as a consequence of the Goddard-Kent-Olive construction \cite{GKO}, $\A_{\Vir,c}$ can be realized as a coset for an appropriate inclusion of loop group nets \cite{KL04}and it turns out to be completely rational. Accordingly Theorem \ref{abstractfusionKtheory} applies to all Virasoro nets with $c<1$ and, in fact, to all conformal nets with $c<1$. These are classified in \cite{KL04}.
\end{example}

\begin{example}
\label{Example:SuperVirasoroNets} The even (Bose) subnet of a $N=1$ super-Virasoro net with central charge $c<3/2$ can be realized as a coset of an inclusion of loop group nets and it turns out to be completely rational \cite[Sect. 6]{CKL}. Accordingly, the even subnets of the superconformal nets with $c<3/2$ are all completely rational and admit a superconformal tensor product so that Theorem \ref{abstractfusionKtheory} applies. These conformal nets have been classified in \cite[Sect.7]{CKL}. 
Similarly, the even subnet of a $N=2$ super-Virasoro net with central charge $c<3$ can be realized as a coset of an inclusion of loop group nets and it turns out to be completely rational \cite[Sect.5]{CHKLX}. Accordingly, the even subnets of the $N=2$ superconformal nets with $c<3$ are all completely rational and admits a superconformal tensor product so that Theorem \ref{abstractfusionKtheory} applies. These conformal nets have been classified in \cite[Sect.6]{CHKLX}. 
\end{example} 

\begin{example}
For every subfactor $N\subset M$ with Jones index $[M:N]<4$, M. Bischoff has constructed in \cite{Bischoff2015} a completely rational conformal net 
$\A_{N\subset M}$ whose representation category is braided tensor equivalent to the quantum double $D(N\subset M)$ and has shown the existence 
of vertex operator algebras $V_{N \subset M}$ with the analogous property.  The nets $\A_{N\subset M}$ are obtained from loop group nets by taking cosets, irreducible local extensions and tensor products. Accordingly, they all admit a superconformal tensor product and  Theorem \ref{abstractfusionKtheory} applies. Note that, for any of these nets, the DHR fusion ring $\tilde{\mathcal{R}}_{\A_{N \subset M}}$ coincides with the fusion ring
of the corresponding vertex operator algebra  $V_{N \subset M}$.
\end{example}

\begin{example} Let $n$ be a positive integer and let $\A_{\Vir,\frac12}^{\otimes n}$ be the conformal net with central charge $n/2$ defined as the tensor product of $n$ copies of the Virasoro net $\A_{\Vir,\frac12}$. Then, $\A_{\Vir,\frac12}^{\otimes n}$ is the tensor product of completely rational conformal nets admitting a superconformal tensor product and hence it is a completely rational conformal net admitting a superconformal tensor product. A conformal net $\A$ is said to be framed if it is an irreducible local extension of $\A_{\Vir,\frac12}^{\otimes n}$ for some positive integer $n$ 
\cite[Sect.4]{KL06}, see also \cite{KS}. Accordingly, every framed conformal net is completely rational and admits a superconformal tensor product so that Theorem \ref{abstractfusionKtheory} applies. Remarkable examples of framed conformal nets are the moonshine net $\A^\natural$ constructed by Kawahigashi and Longo in \cite{KL06}, see also \cite[Thm.8.15]{CKLW} and whose automorphism group is the monster group 
$\mathbb{M}$, and the even shorter net $\A_{VB^\natural_{(0)}}$ constructed in \cite[Thm.8.16]{CKLW}.
\end{example}

All the above examples show that the description of the representation theory of CFTs in terms of K-theory and noncommutative geometry goes far beyond the realm of loop groups. It would be very interesting if some of the above examples also admitted a topological description in terms of twisted K-theory in the spirit of FHT. This would give \eg a twisted K-theory description of the discrete series representations of $\Diff$ and of the representation theory of coset models, \cf \cite[page 2013]{EG2009} and \cite[page 323]{EG2013}. To this end a more direct and clear relation of the results in this paper in the case of loop groups and the FHT work is probably needed.

We end this section with a comment on the case of disconnected compact Lie groups, which are covered in the FHT setting through the analysis of twisted loop groups and their positive-energy representations, \cf \cite{FHT2011b}, but not in our present analysis. The point is that the usual definition of loop group nets generalizes to non-simply connected compact Lie groups but apparently does not generalize in a natural way to disconnected compact Lie groups, \cf \cite{Henriques2015,Henriques2016,Verril,Was2010b}. Note that in \cite{Verril,Was2010b} a fusion product on the representations of certain twisted loop groups is constructed from the point of view of subfactor theory through Connes fusions. The results there indicate that there is no conformal net associated to a twisted loop group. Rather the representations of the twisted loop group should be considered as twisted (soliton) 
representations of the corresponding untwisted loop group net and should be related to the representation theory of orbifold models which, presently, are not covered by our analysis.

\bigskip

\noindent\textbf{Acknowledgements.} 
We would like to thank the Kavli IPMU Japan for invitation and hospitality during the programme ``Supersymmetry in Physics and Mathematics" in March 2014, during which a substantial part of this work was achieved. S.C. would like to thank Andr\'e Henriques for a very useful email correspondence on the loop group nets associated to non-simply connected compact Lie groups and for explanations on his works \cite{Henriques2015,Henriques2016}. We furthermore thank the referee for helpful suggestions on a previous version of the manuscript.

\end{document}